\documentclass[11pt]{article}
\usepackage[T1]{fontenc}
\usepackage[margin=1in]{geometry}
\usepackage{biolinum}
\usepackage{booktabs}
\usepackage{libertineMono}
\usepackage[utf8]{inputenc}
\usepackage{amsfonts,amsmath,amssymb,amsthm}
\usepackage{mathtools}
\usepackage{array}
\usepackage{pdflscape}
\usepackage{cite}
\usepackage{microtype}
\usepackage[colorlinks=true,linkcolor=blue,anchorcolor=blue,citecolor=red,urlcolor=magenta]{hyperref}
\usepackage{tikz}
\usetikzlibrary{decorations.pathreplacing}

\newtheorem{theorem}{Theorem}[section]
\newtheorem{lemma}[theorem]{Lemma}
\newtheorem{corollary}[theorem]{Corollary}
\newtheorem{conjecture}[theorem]{Conjecture}
\newtheorem{proposition}[theorem]{Proposition}
\newtheorem{definition}[theorem]{Definition}
\newtheorem{example}[theorem]{Example}

\newcommand{\diag}{\operatorname{diag}}

\newcommand{\one}{\mathbf{1}}
\newcommand{\0}{\mathbf{0}}

\newcommand\range{\mathop{\rm range}\nolimits}
\newcommand{\spread}{\operatorname{sp}}

\newcommand{\supp}{\operatorname{supp}}
\newcommand{\tr}{\operatorname{Tr}}
\newcommand{\Gc}{\overline G}

\def\R{\mathbb{R}}

\def\b{\mathbf{b}}
\def\d{\mathbf{d}}
\def\e{\mathbf{e}}
\def\m{\mathbf{m}}
\def\p{\mathbf{p}}
\def\vv{\mathbf{v}}
\def\w{\mathbf{w}}
\def\x{\mathbf{x}}
\def\y{\mathbf{y}}
\def\z{\mathbf{z}}

\title{The Laplacian $S_{n,n}$ conjecture is true}
\newcommand{\authemail}[1]{{\footnotesize\ttfamily #1}}
\author{Nathaniel Johnston\thanks{\scriptsize Department of Mathematics \& Computer Science, Mount Allison University. \authemail{njohnston@mta.ca}}}

\date{September 22, 2026}

\begin{document}

\maketitle

\begin{abstract}
    The ``$S_{n,n}$ conjecture'' asserts that there does not exist a simple graph on $n$ vertices with Laplacian spectrum $\{0,1,2,\ldots,n-1\}$ for any integer $n \geq 2$. This conjecture has already been proved when $2 \leq n \leq 15$ and when $n \geq 6,649,688,933$. We prove all remaining cases and thus establish that the conjecture is true.
\end{abstract}

\section{Introduction}\label{sec:intro}

A graph is \emph{Laplacian integral} if every eigenvalue of its Laplacian matrix is an integer. Laplacian integral graphs have received significant interest for several decades (see \cite{GMS90,GM94,Kir07} and the references therein, for example). It was asked in \cite{FKMN05} which Laplacian integral graphs have distinct Laplacian eigenvalues. Since every graph on $n$ vertices has its Laplacian eigenvalues in $[0,n]$, if a Laplacian integral graph has distinct eigenvalues, then its spectrum must consist of all members of $\{0,1,\ldots,n\}$ except for one.

Existence of such a graph when the missing eigenvalue is one of $0$, $1$, $2$, $\ldots$, $n-1$ was completely answered in \cite{FKMN05}. However, the case of $n$ being the missing eigenvalue (i.e., the question of whether or not there exists a graph with Laplacian spectrum $S_{n,n} := \{0,1,2,\ldots,n-1\}$) was left open, and it was conjectured that no such graph exists:

\begin{conjecture}[$S_{n,n}$ conjecture~\cite{FKMN05}]\label{conj:Snn}
    Suppose $n \geq 2$ is an integer. Then there does not exist a simple graph on $n$ vertices with Laplacian spectrum $S_{n,n}$.
\end{conjecture}

This conjecture has already been proved for many different graph orders. It was proved in \cite{FKMN05} when $n \leq 11$, when $n$ is prime, and when $n \equiv 2 \pmod{4}$ or $n \equiv 3 \pmod{4}$. It was proved in \cite{GN13} when $n \geq 6,649,688,933$, and it was proved in \cite{JPV26} when $n = 12$. Our main contribution is to prove the conjecture for all remaining values of $n$.

\subsection{Proof overview}\label{sec:proof_overview}

Our proof of Conjecture~\ref{conj:Snn} proceeds in two main steps: one that handles graphs of order $n \leq 31$ (Section~\ref{part:low_n}), and one that handles graphs of order $n \geq 32$ (Section~\ref{part:high_n}). These two steps use very different techniques and can be read completely independently from each other.

The proof for graphs of order $n \leq 31$ is very computational. It develops a family of linear systems in nonnegative variables that describe necessary conditions for existence of a counterexample to Conjecture~\ref{conj:Snn}. By providing roughly 10,000 infeasibility certificates (i.e., linear programming duality certificates), we can computationally verify the conjecture in this case. A more in-depth overview of the proof of the $n \leq 31$ case is provided in Section~\ref{sec:low_n_computational}.

The proof for graphs of order $n \geq 32$ is much more mathematically involved. We define a family of graphs that we call \emph{eigensimple} in Definition~\ref{defn:eigensimple}, which straightforwardly contains every counterexample to the $S_{n,n}$ conjecture (if any exist). We show that eigensimple graphs on $n \geq 32$ vertices don't exist by developing a series of (quite technical) inequalities that they must satisfy, and then presenting sum-of-squares certificates that show that at least one of them must fail. A much more in-depth overview of this case is provided in Section~\ref{sec:high_n_overview}.

We note that both proofs have some ``wiggle room''. While we only applied the proof technique from Section~\ref{part:low_n} to $n \leq 31$, it could conceivably also be used to prove the conjecture when $n = 32$, albeit at a greatly increased computational cost. Similarly, the inequalities that drive the proof technique from Section~\ref{part:high_n} could likely be strengthened (at the expense of being significantly uglier and harder to work with) to prove the conjecture for some smaller values of $n$. We chose $n = 32$ as the breakpoint between our methods to try to balance the computational requirements of the proof of Section~\ref{part:low_n} and the general theoretical complexity of the proof of Section~\ref{part:high_n}.

\subsection{Consequences of our results}\label{sec:consequences}

Since we prove nonexistence of a wider class of graphs than just graphs with Laplacian spectrum $S_{n,n}$ when $n \geq 32$, our results also have consequences for some other open questions about Laplacian integral graphs. For example, following the notation of~\cite{HT23,HT25}, for integers $1 \leq i < j \leq n$ and $m \in \{1,2,\ldots,n\} \setminus \{i,j\}$ we define the multiset
\begin{align*}
    S_{\{i,j\}_{n}^{m}} := \{0,1,\ldots,m-1,m,m,m+1,\ldots,n-1,n\} \setminus \{i,j\}.
\end{align*}

It was conjectured in \cite[Conjecture~6.4]{HT23} that, other than some examples of orders $6$ and $8$, there does not exist a graph with Laplacian spectrum of the form $S_{\{i,n\}_{n}^{m}}$. Two more graphs (of order $9$) with Laplacian spectra of this form were found in \cite{JPV26}, falsifying the conjecture. However, our results (in particular, the upcoming Theorem~\ref{thm:no_large_eigensimple}) immediately prove the conjecture in the case when $2 \leq i,m \leq n-2$, $i \neq m$, and $n \geq 32$. Consequently, any further counterexamples to the conjecture must have $i \in \{1,n-1\}$, $m\in\{1,n-1\}$, or $n \leq 31$.

Finally, our results also establish several conjectures that are known to follow from the $S_{n,n}$ conjecture. Some examples include \cite[Conjecture~5.4]{HKT22}, the $m = n$ half of \cite[Conjecture~6.5]{HT23}, and the fact that for each $1 \leq i < n$ there exists at most one order-$n$ graph (up to isomorphism) with Laplacian spectrum $\{0,1,\ldots,n\} \setminus \{i\}$ \cite{FKMN05}.

One natural new open problem arises from our work as well: which graphs are eigensimple? Theorem~\ref{thm:no_large_eigensimple} immediately implies that there are only finitely many of them, since they all have $31$ or fewer vertices, but it does not tell us what they all are.

\section{Preliminaries}\label{sec:preliminaries}

We start by introducing some notation and recalling standard results concerning Laplacian integral graphs.

\subsection{Notation and terminology}\label{sec:notation}

All graphs that we consider are finite, simple, and unweighted. The complement of a graph $G$ is denoted by $\Gc$. Write $V(G)$ and $E(G)$ for the vertex set and edge set of $G$, respectively, and let $n := |V(G)|$ denote the order of $G$. For a vertex $v \in V(G)$, write $N_G(v)$ and $d_G(v)$ for the neighborhood of $v$ and degree of $v$, respectively. We use $G[S]$ to denote the subgraph of $G$ induced by $S \subseteq V(G)$. A \emph{cut vertex} of a connected graph is a vertex whose deletion would result in that graph becoming disconnected. A \emph{dominating set} is a set $S \subseteq V(G)$ such that every vertex outside $S$ has a neighbor in $S$.

The \emph{Laplacian matrix} $L(G)$ is the square matrix with rows and columns indexed by $V(G)$, with $[L(G)]_{v,v} = d_G(v)$ for all $v \in V(G)$ and $[L(G)]_{v,w} = -1$ for all $\{v,w\} \in E(G)$ (and $[L(G)]_{v,w} = 0$ otherwise). The \emph{Laplacian eigenvalues} of $G$ are the eigenvalues of $L(G)$, which we list in nondecreasing order, with multiplicity, as $0 = \lambda_1(G) \leq \lambda_2(G) \leq \cdots \leq \lambda_n(G)$ (we note that it is straightforward to show that $L(G)$ is symmetric and diagonally dominant, so it is positive semidefinite). If the graph $G$ is clear from context, we use the abbreviations
\begin{align*}
    L & := L(G), & N(v) & := N_G(v), & d(v) & := d_G(v),\\
    \overline{L} & := L(\Gc), & \overline{N}(v) & := N_{\Gc}(v), & \overline{d}(v) & := d_{\Gc}(v).
\end{align*}

We define a \emph{$1$-eigenvector} of a graph $G$ to be an eigenvector of $L(G)$ corresponding to the eigenvalue $1$. A \emph{unit vector} $\x$ is a vector with Euclidean norm $\|\x\| = 1$. All vectors that we use should be interpreted as column vectors when their shape is relevant, and they are denoted by lowercase bold letters like $\x$ and $\y$. Their coordinates are denoted by the corresponding unbolded letters with subscripts like $x_v$ for the entry of $\x$ indexed by $v$. We denote the all-ones vector by $\one$, the standard basis vector corresponding to index $v$ by $\e_v$, the identity matrix by $I$, and the all-ones matrix by $J := \one\one^T$ (their dimensions will always be clear from context, and are typically equal to $n$).

For a $1$-eigenvector $\x$ of $G$, the equation $L(G)\x = \x$ at a vertex $v \in V(G)$ can be written as
\begin{align}\label{eq:eigenvector_equation}
    x_v = d_G(v)x_v-\sum_{w \in N_G(v)}x_w = \sum_{w \in N_G(v)}(x_v-x_w).
\end{align}
We refer to this as the \emph{eigenvector equation} of $\x$ at $v$. The positive and negative parts of $\x$ are defined coordinatewise by
\begin{align*}
    \x_{+} & := \max\{\x,\0\}, & \x_{-} & := \max\{-\x,\0\}, & \x & = \x_{+} - \x_{-}.
\end{align*}
We write
\begin{align*}
    \max(\x) & := \max_v\{x_v\}, & \min(\x) & := \min_v\{x_v\}, & \range(\x) & := \max(\x) - \min(\x).
\end{align*}
If $\x$ has both positive and negative entries, we define
\begin{align*}
    \spread(\x) := \frac{\min\{\max(\x),-\min(\x)\}}{\range(\x)}
\end{align*}
(the notation ``$\spread(\x)$'' refers to the fact that this quantity gives us information about how spread out the entries of $\x$ are). Notice that $0 < \spread(\x) \leq 1/2$ and $\spread(-\x) = \spread(\x)$. The magnitudes of the two extreme entries of $\x$ are $\spread(\x)\range(\x)$ and $(1-\spread(\x))\range(\x)$, in some order.

We denote the set of indices of nonzero entries of $\x$ by $\supp(\x)$, the restriction of $\x$ to an index set $S$ by $\x_S$, and the diagonal matrix with diagonal entries from $\x$ (in the same order) by $\diag(\x)$. All inequalities involving matrices or vectors are meant entrywise.

\subsection{Some standard facts about Laplacian integral graphs}\label{sec:standard_facts}

Here we list some standard and well-known facts about Laplacian matrices and Laplacian integral graphs (i.e., graphs $G$ for which the eigenvalues of $L(G)$ are all integers) that we will use repeatedly, and often without specific citation.

Every Laplacian matrix is positive semidefinite and has all row and column sums equal to $0$. The multiplicity of its eigenvalue $0$ is the number of connected components of $G$. The vector $\one$ is an eigenvector corresponding to the eigenvalue $0$ of $L(G)$, so every eigenvector $\x$ corresponding to a nonzero eigenvalue of $L(G)$ satisfies $\x \perp \one$. The Laplacian of a complement of an order $n$ graph satisfies
\begin{align*}
    L(\Gc) = nI - J - L(G).
\end{align*}
In particular, if $G$ has spectrum $S_{n,n}$ then so does $\Gc$, and in this case both $G$ and $\Gc$ are connected.

The following lemma provides some known restrictions on what a graph with spectrum $S_{n,n}$ would have to look like; proofs of these facts can all be found in \cite{FKMN05}.

\begin{lemma}\label{lem:Snn_basic_conditions}
    Let $n \geq 2$ be an integer and let $G$ be a simple graph on $n$ vertices with Laplacian spectrum $S_{n,n}$. Then:
    \begin{itemize}
        \item[(a)] $2 \leq d(v) \leq n-3$ for all $v \in V(G)$.
        
        \item[(b)] $n \equiv 0 \pmod{4}$ or $n \equiv 1 \pmod{4}$.
        
        \item[(c)] $n$ is not prime.
    \end{itemize}
\end{lemma}

The following lemma tells us about the sums of powers of Laplacian eigenvalues of $G$ and degrees of vertices in $G$, respectively. These quantities are sometimes called the \emph{Laplacian moments} and \emph{degree power sums} of $G$, respectively; see \cite{PZJP10}, where they are normalized by dividing by $n$. We note that all of these formulas are trivial to prove (and again, most of them are stated in \cite{FKMN05} and other references on this problem).

\begin{lemma}\label{lem:Snn_laplacian_moments}
    Let $n \geq 2$ be an integer and let $G$ be a simple graph on $n$ vertices with Laplacian spectrum $S_{n,n}$. Then:
    \begin{itemize}
        \item[(a)] $\displaystyle \tr\big(L(G)\big) = n(n-1)/2$.
        
        \item[(b)] $\displaystyle \tr\big(L(G)^2\big) = n(n-1)(2n-1)/6$.
              
        \item[(c)] $\displaystyle \sum_{v \in V(G)} d(v) = \frac{n(n-1)}{2}$.
        
        \item[(d)] $\displaystyle \sum_{v \in V(G)} d(v)^2 = \frac{n(n-1)(n-2)}{3}$.
    \end{itemize}
\end{lemma}

\section{Part 1: Graphs of low order}\label{part:low_n}

This section is devoted to the proof that there does not exist a graph of order $2 \leq n \leq 31$ that has Laplacian spectrum equal to $S_{n,n}$. Since this is already known when $2 \leq n \leq 12$ \cite{JPV26}, when $n$ is prime, and when $n$ is congruent to $2$ or $3$ mod $4$ \cite{FKMN05}, it suffices to prove the claim for $n \in \{16,20,21,24,25,28\}$.

\subsection{An overview of the low-order proof}\label{sec:low_n_computational}

While our method is computational, it is very much not by na\"ive brute-force, which is completely infeasible for graphs of order as large as $28$ (e.g., recently all Laplacian integral graphs were computed for $n \leq 13$ \cite{JPV26}, and it seems unlikely that this enumeration will be pushed significantly higher any time soon).

Instead, our search is based on considering the degree sequence that a graph with Laplacian spectrum $S_{n,n}$ must have (an idea that was already explored, though less thoroughly, in \cite{FKMN05}). More specifically, our overall approach consists of the following steps:

\begin{itemize}
  \item[(1)] Derive necessary conditions on the degree sequence of a graph $G$ with Laplacian spectrum $S_{n,n}$.
  
  \item[(2)] For each degree sequence that is not ruled out by these necessary conditions, build a linear system in nonnegative variables that encodes the diagonal entries of the first four powers of $L(G)$ (we will call these quantities the \emph{local Laplacian moments} of $G$).
  
  \item[(3)] Prove that each of these linear systems is infeasible via linear programming duality.
\end{itemize}

We emphasize that our verification is exact. For example, there are no numerical eigenvalue computations. Numerical linear programming was used to \emph{discover} the dual certificates in step~(3) above, but they are then \emph{verified} exactly via arbitrary-precision integer arithmetic.

\subsection{Local Laplacian moments}\label{sec:local_moments}

Thanks to Lemma~\ref{lem:Snn_basic_conditions}(a), our search for an order-$n$ graph $G$ with Laplacian spectrum $S_{n,n}$ can search over degree sequences $(d_1, d_2, \ldots, d_n)$ satisfying
\[
    2 \leq d_1 \leq d_2 \leq \cdots \leq d_n \leq n-3
\]
and the equations from parts (c) and (d) of Lemma~\ref{lem:Snn_laplacian_moments}. Furthermore, complementation sends this degree sequence to $(n-1-d_n, n-1-d_{n-1}, \ldots, n-1-d_1)$, and since $G$ has Laplacian spectrum $S_{n,n}$ if and only if $\Gc$ does, it is enough to consider just one representative from each of these complementary pairs of degree sequences.

The observations that we have made so far constrain the degrees of the vertices in $G$, but they say nothing about \emph{which} vertices of given degrees are adjacent. By evaluating the diagonal entries of powers of $L(G)$ individually (rather than summing them as in parts~(a) and~(b) of Lemma~\ref{lem:Snn_laplacian_moments}), we can get ``local'' information that lets us bound the degrees of adjacent vertices. The corresponding quantities for the adjacency matrix $A(G)$ have been investigated under the name \emph{local spectral moments} \cite{REG07}. We similarly call these diagonal entries of powers of the Laplacian matrix $[L(G)^k]_{v,v}$ the \emph{local Laplacian moments} of $G$ at $v$.

In order to compute entries of powers of $L(G)$, we first need some helper variables that describe the adjacencies of $G$. For distinct vertices $v,w \in V(G)$, we define
\begin{align}\label{eq:na_defn}
    n_{v,w} := |N(v) \cap N(w)| \qquad \text{and} \qquad a_{v,w} := \begin{cases}
        1 & \text{if } \{v,w\} \in E(G),\\
        0 & \text{otherwise},
    \end{cases}
\end{align}
(we note that the notation $a_{v,w}$ refers to the fact that it is the $(v,w)$-entry of the adjacency matrix $A(G)$).

\begin{lemma}\label{lem:laplacian_local_moments}
    Let $n \geq 2$ be an integer and let $G$ be a simple graph on $n$ vertices with Laplacian spectrum $S_{n,n}$. Let $v,w \in V(G)$ be distinct, and let $n_{v,w}$ and $a_{v,w}$ be as in Equation~\eqref{eq:na_defn}. Then:
    \begin{itemize}
        \item[(a)] $\displaystyle \big[L(G)^2\big]_{v,v} = d(v)^2 + d(v)$.
        
        \item[(b)] $\displaystyle \big[L(G)^3\big]_{v,v} = d(v)^3 + 2d(v)^2 + \sum_{x \neq v}a_{v,x}\big(d(x)-n_{v,x}\big)$.
        
        \item[(c)] $\displaystyle \big[L(G)^4\big]_{v,v} = \big(d(v)^2+d(v)\big)^2 + \sum_{x \neq v}\big(n_{v,x}-\big(d(v)+d(x)\big)a_{v,x}\big)^2$.
        
        \item[(d)] $\displaystyle \sum_{x \neq v}n_{v,x} = \sum_{x \in N(v)}\big(d(x)-1\big)$.
    \end{itemize}
\end{lemma}

\begin{proof}
    Part~(a) follows simply by multiplying $L(G)$ by itself, as does the formula
    \begin{align}\label{eq:L2_offdiag}
        \big[L(G)^2\big]_{v,w} = n_{v,w} - \big(d(v)+d(w)\big)a_{v,w}.
    \end{align}
    It follows that
    \begin{align*}
        \big[L(G)^3\big]_{v,v} & = d(v)\big[L(G)^2\big]_{v,v} - \sum_{x \neq v}a_{v,x}\big[L(G)^2\big]_{v,x} \\
        & = d(v)\big(d(v)^2 + d(v)\big) - \sum_{x \neq v}a_{v,x}\big(n_{v,x} - \big(d(v)+d(x)\big)a_{v,x}\big),
    \end{align*}
    which simplifies to the formula provided in part~(b). Similarly,
    \begin{align*}
        \big[L(G)^4\big]_{v,v} & = \sum_{x \in V(G)}\big[L(G)^2\big]_{v,x}^2 = \big[L(G)^2\big]_{v,v}^2 + \sum_{x \neq v}\big[L(G)^2\big]_{v,x}^2,
    \end{align*}
    which gives part~(c) once we plug in the formulas from part~(a) and Equation~\eqref{eq:L2_offdiag}. Finally, both sides of part~(d) count the number of length-$2$ walks in $G$ that start at $v$ but do not end at $v$.
\end{proof}

\subsection{Setting up the linear systems}\label{sec:lown_linear_system}

We now describe a relaxation of the problem ``there exists a simple graph $G$ of order $n$ with Laplacian spectrum $S_{n,n}$'' that takes the form of a system of linear equations in nonnegative real variables. Verification that these linear systems have no solutions when $n \in \{16, 20, 21, 24, 25, 28\}$ is what will show that there is no simple graph $G$ with Laplacian spectrum $S_{n,n}$ for these orders.

Suppose that $G$ is an order-$n$ graph with Laplacian spectrum $S_{n,n}$. For each $2 \leq d \leq n-3$, let $m_d := |\{v \in V(G) : d(v) = d\}|$ be the multiplicity of the degree $d$ in the degree sequence of $G$, and write $\m := (m_2,m_3,\ldots,m_{n-3})$. We construct a linear system for each such vector that is not ruled out by Lemma~\ref{lem:Snn_basic_conditions}(a) and parts~(c) and~(d) of Lemma~\ref{lem:Snn_laplacian_moments}. The variables in this linear system will be denoted by $y_{d,k}$ for integers $d$ and $k$, and $z_{a,b,c,d}$ for integers $a$, $b$, $c$, and $d$. These variables will all be nonnegative.

\subsubsection{The first set of variables}\label{sec:lown_linear_system_y}

We start by describing the variables $y_{d,k}$ and illustrating what constraints they must satisfy. Let $\{\w^{(0)}, \w^{(1)}, \ldots, \w^{(n-1)}\}$ be a real orthonormal eigenbasis of $L(G)$ with $\w^{(k)}$ corresponding to eigenvalue $k$ for all $0 \leq k \leq n-1$ (choose $\w^{(0)}=\one/\sqrt{n}$). For $2 \leq d \leq n-3$ and $0 \leq k \leq n-1$, we define
\begin{equation}\label{eq:ydef}
    y_{d,k} := \sum_{\substack{v \in V(G) \\ d(v)=d}} \big(w_v^{(k)}\big)^2,
\end{equation}
which measures how much of the unit Laplacian eigenvector of $G$ corresponding to eigenvalue $k$ is concentrated on degree-$d$ vertices. Notice that $y_{d,0}=m_d/n$, so these variables are determined by $\m$ and thus do not need to be included as variables in the linear system.

\begin{lemma}\label{lem:y_dk_constraints}
    Let $n \geq 2$ be an integer and let $G$ be a simple graph on $n$ vertices with Laplacian spectrum equal to $S_{n,n}$. Let $y_{d,k}$ be as in Equation~\eqref{eq:ydef}. Then:
    \begin{itemize}
        \item[(a)] $\displaystyle \sum_{d=2}^{n-3}y_{d,k} = 1$ for all integers $0 \leq k \leq n-1$.
        
        \item[(b)] $\displaystyle n\sum_{k=1}^{n-1}y_{d,k}=(n-1)m_d$ for all integers $2 \leq d \leq n-3$.
        
        \item[(c)] $\displaystyle \sum_{k=1}^{n-1}k y_{d,k} = dm_d$ for all integers $2 \leq d \leq n-3$.
        
        \item[(d)] $\displaystyle \sum_{k=1}^{n-1}k^2 y_{d,k} = (d^2+d)m_d$ for all integers $2 \leq d \leq n-3$.
    \end{itemize}
\end{lemma}

\begin{proof}
    Part~(a) follows from the fact that $\sum_{d=2}^{n-3}y_{d,k} = \big\|\w^{(k)}\big\|^2$, and each $\w^{(k)}$ comes from an orthonormal basis and is thus a unit vector.

    Now notice that the fact that $\{\w^{(0)}, \w^{(1)}, \ldots, \w^{(n-1)}\}$ is an orthonormal basis tells us that $\sum_{k=0}^{n-1}\big(w_v^{(k)}\big)^2 = 1$ for all $v \in V(G)$. Since $\w^{(0)} = \one/\sqrt{n}$, removing the $k = 0$ term from this sum gives
    \[
        \sum_{k=1}^{n-1}\big(w_v^{(k)}\big)^2=1-\frac{1}{n}=\frac{n-1}{n}.
    \]
    Summing this identity over the $m_d$ vertices of degree $d$ and multiplying by $n$ gives part~(b).

    To see parts~(c) and (d), we first use the fact that the eigenvalues of $L(G)$ are $0$, $1$, $\ldots$, $n-1$, with corresponding orthonormal eigenbasis $\{\w^{(0)}, \w^{(1)}, \ldots, \w^{(n-1)}\}$, to see that if $r \geq 1$ is an integer then
    \begin{align}\label{eq:yk_powers_L}
        \sum_{\substack{v \in V(G)\\d(v)=d}}\big[L(G)^r\big]_{v,v} = \sum_{\substack{v \in V(G) \\ d(v)=d}} \sum_{k=1}^{n-1}k^r \big(w_v^{(k)}\big)^2 = \sum_{k=1}^{n-1}k^r y_{d,k}.
    \end{align}
    When $r = 1$, the fact that $\big[L(G)\big]_{v,v} = d(v)$ gives part~(c). When $r = 2$, Lemma~\ref{lem:laplacian_local_moments}(a) gives part~(d).
\end{proof}

\subsubsection{The second set of variables}\label{sec:lown_linear_system_z}

The four sets of equations of Lemma~\ref{lem:y_dk_constraints} give necessary conditions for the existence of an order $n$ graph with Laplacian spectrum $S_{n,n}$, but alone they are not enough to show that no such graph exists for the values of $n$ that we are interested in. For this reason, we now introduce a second set of nonnegative variables that can be used to construct additional necessary linear equations.

For integers $a \in \{0,1\}$ and $2 \leq c \leq d \leq n-3$ (we will specify the possible values of $b$ shortly), we define $z_{a,b,c,d}$ to be the number of unordered pairs $\{v,w\}$ of distinct vertices of $G$ such that
\[
    a_{v,w} = a, \qquad n_{v,w} = b, \qquad d(v) = c, \quad \text{and} \quad d(w) = d,
\]
where $a_{v,w}$ and $n_{v,w}$ are as in Equation~\eqref{eq:na_defn}. We note that $b$ is bounded by the degrees of $v$ and $w$: after deleting those two vertices, their neighborhoods have sizes $c-a$ and $d-a$ inside a set of $n-2$ vertices. This gives
\begin{align}\label{eq:allowed_b}
    \max\{0, c+d-(n-2+2a)\} \leq b \leq \min\{c, d\} - a.
\end{align}

Although the $z$ variables count pairs in a graph and are thus nonnegative integers, we do not require them to be integers in the linear system that we are constructing (which causes no problems, since this linear system is just a necessary condition for existence of a simple graph with Laplacian spectrum $S_{n,n}$, not a sufficient one).

To write the upcoming linear equations involving the $z_{a,b,c,d}$ variables compactly, we introduce a notation for summing over all ordered pairs $(v,w)$ of distinct vertices with $d(v) = d$. For any function $f$ of $a$, $b$, $c$, and $d$, we define
\begin{align*}
    Z_d(f) := \sum_{c<d}\sum_{a,b}f(a,b,c,d)z_{a,b,c,d} + \sum_{c>d}\sum_{a,b}f(a,b,c,d)z_{a,b,d,c} + 2\sum_{a,b}f(a,b,d,d)z_{a,b,d,d},
\end{align*}
which is linear in the $z_{a,b,c,d}$ variables. We note that, even though the degrees $c$ and $d$ in the subscript of $z_{a,b,c,d}$ are always sorted with $c \leq d$, there is no such restriction on $f$. Both sums over $c$ are restricted to $2 \leq c \leq n-3$, and the sums over $a$ and $b$ are restricted to $a \in \{0,1\}$ and the values of $b$ from Inequality~\eqref{eq:allowed_b}. It follows that
\[
    Z_d(f) = \sum_{\substack{v \in V(G)\\d(v)=d}}\sum_{\substack{w \in V(G)\\w\neq v}} f\big(a_{v,w},n_{v,w},d(w),d(v)\big).
\]

With this notation in hand, we are in a position to state the linear equations that the $z_{a,b,c,d}$ variables must satisfy:

\begin{lemma}\label{lem:z_constraints}
    Let $n \geq 2$ be an integer and let $G$ be a simple graph on $n$ vertices with Laplacian spectrum equal to $S_{n,n}$. Then the variables $z_{a,b,c,d}$ satisfy the following:
    \begin{itemize}
        \item[(a)] $Z_d(a) = dm_d$ for all integers $2 \leq d \leq n-3$.
        
        \item[(b)] $Z_d\big(b-a(c-1)\big) = 0$ for all integers $2 \leq d \leq n-3$.
    
        \item[(c)] $\displaystyle \sum_{k=1}^{n-1}k^3 y_{d,k} - Z_d\big(a(c-b)\big) = (d^3+2d^2)m_d$ for all integers $2 \leq d \leq n-3$.
        
        \item[(d)] $\displaystyle \sum_{k=1}^{n-1}k^4 y_{d,k} - Z_d\big((b-(c+d)a)^2\big) = (d^2+d)^2 m_d$ for all integers $2 \leq d \leq n-3$.

        \item[(e)] $\displaystyle \sum_{a,b}z_{a,b,c,d} = \begin{cases} m_cm_d & \text{if} \ \ c<d\\ m_d(m_d-1)/2 & \text{if} \ \ c=d\end{cases}$ \ for all integers $2 \leq c \leq d \leq n-3$.
    \end{itemize}
\end{lemma}

\begin{proof}
    Part~(a) follows from the formula
    \[
        Z_d(a) = \sum_{\substack{v \in V(G)\\d(v)=d}}\sum_{\substack{w \in V(G)\\w\neq v}} a_{v,w} = \sum_{\substack{v \in V(G)\\d(v)=d}}d(v) = dm_d.
    \]
    Similarly,
    \[
        Z_d\big(b-a(c-1)\big) = \sum_{\substack{v \in V(G)\\d(v)=d}}\left(\sum_{\substack{w \in V(G)\\w\neq v}} n_{v,w} - \sum_{w\in N(v)}\big(d(w)-1\big) \right) = 0,
    \]
    where the final equality follows from Lemma~\ref{lem:laplacian_local_moments}(d). This proves part~(b).
    
    To prove parts~(c) and~(d), simply plug the formulas from parts~(b) and~(c) of
    Lemma~\ref{lem:laplacian_local_moments} into Equation~\eqref{eq:yk_powers_L}. Finally, the left-hand side of part~(e) counts all unordered pairs of distinct (not necessarily adjacent) vertices with degrees $c$ and $d$. There are $m_cm_d$ such pairs when $c < d$, and $m_d(m_d-1)/2$ such pairs when $c = d$.
\end{proof}

Putting this all together gives the following result:

\begin{proposition}\label{prop:necessary_linsys}
    Let $n \geq 2$ be an integer, and let $G$ be a simple graph on $n$ vertices with Laplacian spectrum $S_{n,n}$ and degree-multiplicity vector $\m$. Then there exist nonnegative values of the variables $y_{d,k}$ and $z_{a,b,c,d}$ that simultaneously solve the four families of linear equations from Lemma~\ref{lem:y_dk_constraints} and the five families of linear equations from Lemma~\ref{lem:z_constraints}.
\end{proposition}

We emphasize that Proposition~\ref{prop:necessary_linsys} is just a necessary condition for the existence of a simple graph $G$ on $n$ vertices with Laplacian spectrum $S_{n,n}$. Its converse does not hold in general: a solution of the linear system need not come from a graph. We demonstrate with an example.

\begin{example}\label{exam:n_12_linsys_solve}
    Let $n = 12$ and consider the multiplicities
    \[
        \m = (m_2,m_3,m_4,m_5,m_6,m_7,m_8,m_9) =(3,0,1,1,4,0,0,3)
    \]
    corresponding to the degree sequence $(2,2,2,4,5,6,6,6,6,9,9,9)$. The two tables below list nonnegative values of the $y_{d,k}$ and $z_{a,b,c,d}$ variables that solve the linear system described by Proposition~\ref{prop:necessary_linsys}. All unlisted $y_{d,k}$ with $k \geq 1$, and all unlisted $z_{a,b,c,d}$, equal zero. In particular, since $m_3 = m_7 = m_8 = 0$, we have $y_{3,k} = y_{7,k} = y_{8,k} = 0$ for all $k$. As before, $y_{d,0} = m_d/12$.

    We start with the $y_{d,k}$ values:
    \begin{center}
        \begin{tabular}{c|ccccc}\toprule
            $k$ & $y_{2,k}$ & $y_{4,k}$ & $y_{5,k}$ & $y_{6,k}$ & $y_{9,k}$ \\ \midrule
            1 & $63/80$ & $21/1280$ & $7459/172800$ & $5459/43200$ & $17/640$ \\
            2 & $1$ & $0$ & $0$ & $0$ & $0$ \\
            3 & $59/64$ & $7/128$ & $0$ & $0$ & $3/128$ \\
            4 & $0$ & $101/140$ & $0$ & $39/140$ & $0$ \\
            5 & $0$ & $25/384$ & $11977/17280$ & $2089/8640$ & $0$ \\
            6 & $0$ & $0$ & $0$ & $1$ & $0$ \\
            7 & $0$ & $0$ & $143/1728$ & $1585/1728$ & $0$ \\
            8 & $0$ & $0$ & $0$ & $1$ & $0$ \\
            9 & $0$ & $79/3840$ & $311/3840$ & $0$ & $115/128$ \\
            10 & $0$ & $0$ & $0$ & $0$ & $1$ \\
            11 & $13/320$ & $517/13440$ & $719/43200$ & $31057/302400$ & $513/640$ \\ \bottomrule
        \end{tabular}
    \end{center}
    For example, each row of this table sums to $1$, which verifies that Lemma~\ref{lem:y_dk_constraints}(a) is satisfied. The five columns sum to $11/4$, $11/12$, $11/12$, $11/3$, and $11/4$, respectively. Since these all equal $11m_d/12$, Lemma~\ref{lem:y_dk_constraints}(b) is satisfied. The other two parts of Lemma~\ref{lem:y_dk_constraints} are similarly straightforward (albeit somewhat tedious) to verify.
    
    Next, we list the nonzero $z_{a,b,c,d}$ values:\\[-0.8cm]
    \begin{center}
        \begin{tabular}{cc@{\hspace{1cm}}cc@{\hspace{1cm}}cc@{\hspace{1cm}}cc}\toprule
            $(a,b,c,d)$ & $z_{a,b,c,d}$ & $(a,b,c,d)$ & $z_{a,b,c,d}$ & $(a,b,c,d)$ & $z_{a,b,c,d}$ & $(a,b,c,d)$ & $z_{a,b,c,d}$ \\\midrule
            $(0,0,2,2)$ & $9/4$ & $(0,4,4,5)$ & $1/2$ & $(1,1,2,2)$ & $3/4$ & $(1,3,5,6)$ & $763/540$ \\
            $(0,0,2,5)$ & $763/360$ & $(0,4,4,6)$ & $7/2$ & $(1,1,2,9)$ & $9/4$ & $(1,3,6,9)$ & $29/8$ \\
            $(0,0,2,6)$ & $631/180$ & $(0,4,6,6)$ & $1$ & $(1,1,4,5)$ & $1/2$ & $(1,4,5,9)$ & $3$ \\
            $(0,2,2,4)$ & $3$ & $(0,5,5,6)$ & $5/2$ & $(1,1,4,9)$ & $2$ & $(1,4,6,6)$ & $451/270$ \\
            $(0,2,2,5)$ & $317/360$ & $(0,9,9,9)$ & $3/4$ & $(1,1,6,6)$ & $89/270$ & $(1,5,6,6)$ & $3$ \\
            $(0,2,2,6)$ & $1529/180$& $(1,0,2,9)$ & $9/4$ & $(1,2,4,6)$ & $1/2$ & $(1,5,6,9)$ & $67/8$ \\
            $(0,2,2,9)$ & $9/2$ & $(1,0,5,6)$ & $47/540$ & $(1,3,4,9)$ & $1$ & $(1,6,9,9)$ & $9/4$ \\\bottomrule
        \end{tabular}
    \end{center}
    For example, Lemma~\ref{lem:z_constraints}(e) is satisfied when $c = d = 2$, since
    \[
        \sum_{a,b} z_{a,b,2,2} = z_{0,0,2,2} + z_{1,1,2,2} = \frac{9}{4} + \frac{3}{4} = 3 = m_2(m_2 - 1)/2.
    \]
    The other equations are similarly straightforward to check.
\end{example}

In spite of the above example when $n = 12$, there does not exist a simple graph on $12$ vertices with Laplacian spectrum $S_{n,n}$. Furthermore, the linear system described by Proposition~\ref{prop:necessary_linsys} does not have any nonnegative solutions when $n \in \{16,20,21,24,25,28\}$, which are the values of interest to us. We describe the computational proof of this fact in the next two subsections.

\subsection{Infeasibility certificates}\label{sec:lown_infeasible}

For a fixed order $n$, write the system described by Proposition~\ref{prop:necessary_linsys} as $A_n\x = \b_{\m,n}$, where $\x \geq \0$ is a vector containing the $y_{d,k}$ variables with $k \geq 1$ and the $z_{a,b,c,d}$ variables. The integer matrix $A_n$ depends only on $n$, while the integer vector $\b_{\m,n}$ depends on $n$ as well as the multiplicity vector $\m$.

We certify infeasibility of this nonnegative linear system by using the following well-known consequence of linear programming duality, often called \emph{Farkas' lemma} \cite[Corollary~7.1d]{Sch86}.

\begin{lemma}\label{lem:farkas}
    Suppose that $A$ is an integer matrix and $\b$ and $\vv$ are integer vectors of compatible dimensions. If
    \[
        \vv^T A \geq \0 \qquad \text{and} \qquad \vv^T\b < 0
    \]
    then there does not exist a vector $\x \geq \0$ satisfying $A\x = \b$.
\end{lemma}

\begin{proof}
    Suppose (for the sake of establishing a contradiction) that $\x \geq \0$ is such that $A\x = \b$. Then
    \[
        0 \leq (\vv^T A)\x  = \vv^T\b < 0,
    \]
    which is a contradiction.
\end{proof}

We call the vector $\vv$ an \emph{infeasibility certificate}, or simply a \emph{certificate} for short. Once we have a certificate $\vv$, verifying that the associated linear system does not have a nonnegative solution is as straightforward as checking that $\vv^T A \geq \0$ and $\vv^T\b < 0$, which are computations only involving addition and multiplication of integers.

\subsection{The computation}\label{sec:lown_the_computation}

We now have all of the tools required to describe our computational proof that there does not exist an order $n$ graph with Laplacian spectrum $S_{n,n}$ when $n \in \{16,20,21,24,25,28\}$. For a fixed value of $n$, the computation proceeds as follows:

\begin{itemize}
    \item Enumerate the nondecreasing integer sequences $(d_1, d_2, \ldots, d_n)$ satisfying $2 \leq d_1 \leq d_2 \leq \cdots \leq d_n \leq n-3$ and
    \begin{align*}
        \sum_{i=1}^n d_i = \frac{n(n-1)}{2} \quad \text{and} \quad \sum_{i=1}^n d_i^2 = \frac{n(n-1)(n-2)}{3}.
    \end{align*}
    By Lemma~\ref{lem:Snn_basic_conditions}(a) and parts~(c) and~(d) of Lemma~\ref{lem:Snn_laplacian_moments}, this gets us all degree sequences that we need to check.

    \item For every pair of distinct complementary degree sequences, discard one of them (since $G$ has Laplacian spectrum $S_{n,n}$ if and only if $\Gc$ does). Retain every degree sequence that equals its own complement.

    \item For each remaining degree sequence, we have provided an infeasibility certificate $\vv$ (see \cite{JohCertificates}). Read this infeasibility certificate and verify that
    \[
        \vv^T A_n \geq \0 \qquad \text{and} \qquad \vv^T\b_{\m,n} < 0,
    \]
    where $\m$ is the multiplicity vector of the sequence or its complement.
\end{itemize}

The infeasibility certificates that are required for the third bullet point above were found via extensive computation. However, now that they have been found for all degree sequences when $n \in \{16,20,21,24,25,28\}$, verification of the $S_{n,n}$ conjecture in these cases only requires the computational steps described above.

The number of degree sequences, degree sequence representatives (degree sequences up to complementation), and infeasibility certificates for the $6$ values of $n$ that are of interest to us are provided in Table~\ref{tab:infeasible_counts}. For each order, the verifier checked every entry of $\vv^T A_n$ once for each distinct certificate used. For every representative, it also checked the strict inequality $\vv^T\b_{\m,n} < 0$ for the multiplicity vector of that representative or its complement. In particular, every candidate degree sequence was ruled out, either directly by an infeasibility certificate or by complementation, which immediately gives us the conclusion that we want:

\begin{theorem}\label{thm:certified_small_orders}
    Suppose $n \in \{16,20,21,24,25,28\}$. Then there does not exist a simple graph on $n$ vertices with Laplacian spectrum $S_{n,n}$.
\end{theorem}

\begin{table}[ht]
    \centering
    \begin{tabular}{rrrrr}\toprule
        $n$ & Degree sequences & Representatives & Certificates & Verification time\\\midrule
        16 & 2,249 & 1,135 & 73 & 0.16 seconds \\
        20 & 209,932 & 104,966 & 502 & 3.80 seconds \\
        21 & 675,075 & 337,610 & 581 & 11.7 seconds \\
        24 & 23,479,194 & 11,740,338 & 1,619 & 6.99 minutes \\
        25 & 78,027,831 & 39,014,564 & 1,925 & 24.1 minutes \\
        28 & 2,986,596,342 & 1,493,298,171 & 5,204 & 15.8 hours \\\bottomrule
    \end{tabular}
    \caption{The number of degree sequences, degree sequence representatives (up to complementation), and infeasibility certificates for $n \in \{16,20,21,24,25,28\}$. The final column records the amount of time required for us to verify all certificates of that order on a standard desktop computer.}
    \label{tab:infeasible_counts}
\end{table}

All infeasibility certificates, as well as the code used to verify them, are available for download at \cite{JohCertificates}.

\section{Part 2: Graphs of high order}\label{part:high_n}

This section is devoted to the proof that there does not exist a graph of order $n \geq 32$ that has Laplacian spectrum equal to $S_{n,n}$. The proof of this case is much more mathematically involved, and much less computational, than the proof of the low-order case was.

\subsection{An overview of the high-order proof}\label{sec:high_n_overview}

Before we can describe how our proof of the conjecture for $n \geq 32$ works, we have to introduce a new family of graphs:

\begin{definition}\label{defn:eigensimple}
    We call a simple graph $G$ on $n \geq 3$ vertices \emph{eigensimple} if
    \begin{align*}
        \lambda_2(G) = \lambda_2(\Gc) = 1 \quad \text{and} \quad \lambda_3(G), \lambda_3(\Gc) \geq 2.
    \end{align*}
\end{definition}

In particular, notice that eigensimple graphs are necessarily connected (since $\lambda_2(G) = 1 > 0$) with algebraic connectivity equal to $\lambda_2(G) = 1$, and if $G$ is eigensimple then so is $\Gc$. For $n\ge3$, every simple graph with Laplacian spectrum $S_{n,n}$ is eigensimple (however, we emphasize that eigensimple graphs are \emph{not} defined to be Laplacian integral). Despite how seemingly weak the spectral constraints on eigensimple graphs are, our main result of this section shows that they do not exist when $n \geq 32$. This theorem, once proved, immediately implies Conjecture~\ref{conj:Snn} when $n \geq 32$, since every graph with Laplacian spectrum $S_{n,n}$ is eigensimple:

\begin{theorem}\label{thm:no_large_eigensimple}
    There do not exist eigensimple graphs on $n \geq 32$ vertices.
\end{theorem}

However, there do exist eigensimple graphs on fewer vertices. For example, the cycle $C_6$ with an edge joining two opposite vertices is eigensimple (see Figure~\ref{fig:eigensimple_6_vertex}). Its Laplacian spectrum is $\{0,1,2,3,3,5\}$, and that of its complement is $\{0,1,3,3,4,5\}$. Neither graph is a counterexample to Conjecture~\ref{conj:Snn}, since neither spectrum equals $S_{6,6} = \{0,1,2,3,4,5\}$.

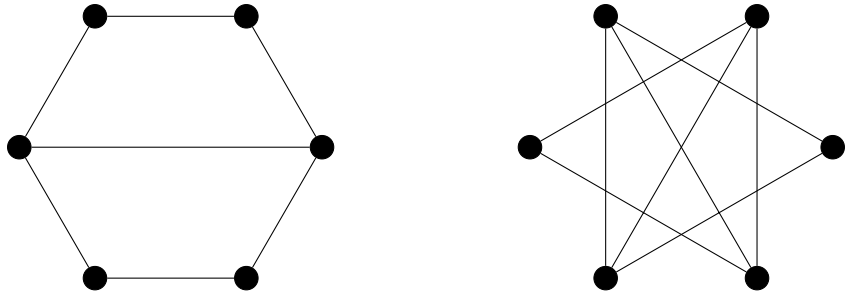
\begin{figure}[htb]
    \centering
    \begin{tikzpicture}[scale=2]
        \node[draw, circle, fill=black, draw=black, inner sep=1.5pt] (v1) at (1,0) {\tiny v};
        \node[draw, circle, fill=black, draw=black, inner sep=1.5pt] (v2) at (0.5,0.866) {\tiny v};
        \node[draw, circle, fill=black, draw=black, inner sep=1.5pt] (v3) at (-0.5,0.866) {\tiny v};
        \node[draw, circle, fill=black, draw=black, inner sep=1.5pt] (v4) at (-1,0) {\tiny v};
        \node[draw, circle, fill=black, draw=black, inner sep=1.5pt] (v5) at (-0.5,-0.866) {\tiny v};
        \node[draw, circle, fill=black, draw=black, inner sep=1.5pt] (v6) at (0.5,-0.866) {\tiny v};
        
        \draw (v1) -- (v2) -- (v3) -- (v4) -- (v5) -- (v6) -- (v1) -- (v4);
    \end{tikzpicture} \qquad \qquad \qquad 
    \begin{tikzpicture}[scale=2]
        \node[draw, circle, fill=black, draw=black, inner sep=1.5pt] (v1) at (1,0) {\tiny v};
        \node[draw, circle, fill=black, draw=black, inner sep=1.5pt] (v2) at (0.5,0.866) {\tiny v};
        \node[draw, circle, fill=black, draw=black, inner sep=1.5pt] (v3) at (-0.5,0.866) {\tiny v};
        \node[draw, circle, fill=black, draw=black, inner sep=1.5pt] (v4) at (-1,0) {\tiny v};
        \node[draw, circle, fill=black, draw=black, inner sep=1.5pt] (v5) at (-0.5,-0.866) {\tiny v};
        \node[draw, circle, fill=black, draw=black, inner sep=1.5pt] (v6) at (0.5,-0.866) {\tiny v};
        
        \draw (v3) -- (v5) -- (v1) -- (v3) -- (v6);
        \draw (v2) -- (v4) -- (v6) -- (v2) -- (v5);
    \end{tikzpicture}
    \caption{An eigensimple graph on $n = 6$ vertices and its complement (which is also eigensimple).}\label{fig:eigensimple_6_vertex}
\end{figure}

Our proof of Theorem~\ref{thm:no_large_eigensimple} is based on comparing unit $1$-eigenvectors $\x$ and $\y$ of $G$ and $\Gc$, respectively. We first show (in Lemma~\ref{lem:extremal_path}) that the vertices where these vectors attain their largest and smallest entries can be chosen to form an induced path $p - q - r - s$ in $G$, where
\begin{align*}
    x_p & = \max(\x), & y_q & = \max(\y), & y_r & = \min(\y), & x_s & = \min(\x).
\end{align*}
Furthermore, $\{q,r\}$ dominates $G$ and $\{p,s\}$ dominates $\Gc$. Figure~\ref{fig:eigensimple_4_path_example} illustrates these properties for the eigensimple graph from Figure~\ref{fig:eigensimple_6_vertex}.

\begin{figure}[htb]
    \centering
    \begin{tikzpicture}[scale=2]
        \node[draw, circle, fill=black, draw=black, inner sep=1.5pt] (v1) at (1,0) {\color{white}$r$};
        \node[draw, circle, fill=black, draw=black, inner sep=1.5pt] (v2) at (0.5,0.866) {\tiny v};
        \node[draw, circle, fill=black, inner sep=1.5pt] (v3) at (-0.5,0.866) {\color{white}$p$};
        \node[draw, circle, fill=black, draw=black, inner sep=1.5pt] (v4) at (-1,0) {\color{white}$q$};
        \node[draw, circle, fill=black, draw=black, inner sep=1.5pt] (v5) at (-0.5,-0.866) {\tiny v};
        \node[draw, circle, fill=black, draw=black, inner sep=1.5pt] (v6) at (0.5,-0.866) {\color{white}$s$};
        
        \draw[black!50!gray] (v1) -- (v2) -- (v3) -- (v4) -- (v5) -- (v6) -- (v1) -- (v4);
        \draw[line width=2.4pt] (v3) -- (v4) -- (v1) -- (v6);
    \end{tikzpicture} \qquad \qquad \qquad 
    \begin{tikzpicture}[scale=2]
        \node[draw, circle, fill=black, draw=black, inner sep=1.5pt] (v1) at (1,0) {\color{white}$r$};
        \node[draw, circle, fill=black, draw=black, inner sep=1.5pt] (v2) at (0.5,0.866) {\tiny v};
        \node[draw, circle, fill=black, inner sep=1.5pt] (v3) at (-0.5,0.866) {\color{white}$p$};
        \node[draw, circle, fill=black, draw=black, inner sep=1.5pt] (v4) at (-1,0) {\color{white}$q$};
        \node[draw, circle, fill=black, draw=black, inner sep=1.5pt] (v5) at (-0.5,-0.866) {\tiny v};
        \node[draw, circle, fill=black, draw=black, inner sep=1.5pt] (v6) at (0.5,-0.866) {\color{white}$s$};
        
        \draw[black!50!gray] (v3) -- (v5) -- (v1) -- (v3) -- (v6);
        \draw[black!50!gray] (v2) -- (v4) -- (v6) -- (v2) -- (v5);
        \draw[line width=2.4pt] (v4) -- (v6) -- (v3) -- (v1);
    \end{tikzpicture}
    \caption{Four vertices ($p$, $q$, $r$, and $s$) that induce a path on four vertices in the graph $G$ from Figure~\ref{fig:eigensimple_6_vertex} and its complement. In $G$ (left) $\{q,r\}$ is a dominating set, while in $\Gc$ (right) $\{p,s\}$ is a dominating set. If we index their Laplacian matrices by vertices in the order $p$, $q$, $r$, $s$, top-right, bottom-left, then (up to sign change) the unique unit $1$-eigenvector of $G$ is $\x = (1, 0, 0, -1, 1, -1)/2$, which has $x_p = \max(\x) = 1/2$ and $x_s = \min(\x) = -1/2$. Similarly, the unique (up to sign change) unit $1$-eigenvector of $\Gc$ is $\y = (-1, 2, -2, 1, 1, -1)/\sqrt{12}$, which has $y_q = \max(\y) = 1/\sqrt{3}$ and $y_r = \min(\y) = -1/\sqrt{3}$. }\label{fig:eigensimple_4_path_example}
\end{figure}
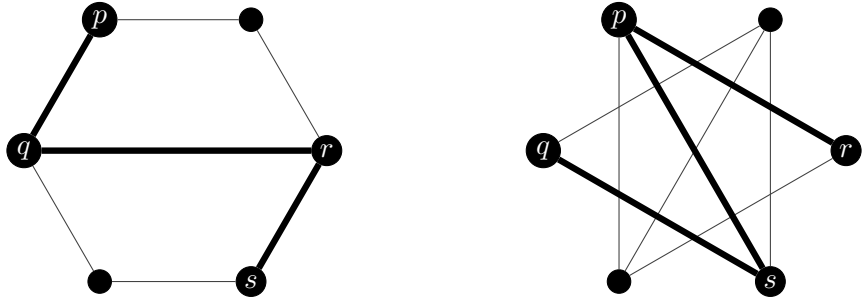

We then (in Lemma~\ref{lem:max_513_bound}) bound the sizes of the extreme entries of each $1$-eigenvector of an eigensimple graph on $n \geq 32$ vertices. The main bound of this type is
\begin{align*}
    \frac{\max(\x)^2}{\|\x_+\|^2} \leq \frac{5}{13} \qquad \text{and} \qquad \frac{\min(\x)^2}{\|\x_-\|^2} \leq \frac{5}{13},
\end{align*}
and the same inequalities hold for $\y$ (we do not claim that this bound is tight, but it is sufficient for our purposes). We similarly prove an inequality of the form
\[
    \range(\x)^2 \leq f(\spread(\x))
\]
in Lemma~\ref{lem:range_bound} (where $f : \R \rightarrow \R$ is an explicit function). Taken together, these inequalities place tight restrictions on how spread out the entries of a $1$-eigenvector of an eigensimple graph can be.

We then proceed in Section~\ref{sec:two_eigenvectors} to develop inequalities that compare the entries of a unit $1$-eigenvector $\x$ of a graph to the entries of a unit $1$-eigenvector $\y$ of its complement. In particular, the entries of these $1$-eigenvectors corresponding to the four distinguished vertices $p$, $q$, $r$, and $s$ from earlier are related to each other, which lets us derive an inequality of the form
\begin{align*}
    \range(\x)^2+\range(\y)^2-\range(\x)^2\range(\y)^2g(\spread(\x),\spread(\y)) \geq 1
\end{align*}
(where $g : \R^2 \rightarrow \R$ is an explicit function; see Lemma~\ref{lem:quartic_inequalities}). We also develop a few other (somewhat more technical) inequalities of this nature that relate the quantities $\range(\x)$, $\range(\y)$, $\spread(\x)$, and $\spread(\y)$.

Some of these inequalities are strongest when $p$, $q$, $r$, and/or $s$ have large degree, while others are strongest when they have small degree. By splitting into a few cases depending on their degrees, we are able to derive a contradiction every time, showing that the eigensimple graph that we started with does not actually exist. More specifically, after splitting into cases according to the vertex degrees, it suffices to prove positivity of five biquadratic polynomials to get the desired contradiction. We give positive definite matrix representations of these polynomials in Appendix~\ref{app:certificates}, which completes the proof.

\subsection{Basic properties of eigensimple graphs}\label{sec:eigensimple_basic}

We start with some basic properties of eigensimple graphs and their unit $1$-eigenvectors.

\begin{lemma}\label{lem:eigensimple_properties}
    Let $G$ be an eigensimple graph on $n \geq 3$ vertices, let $v \in V(G)$, and let $\x, \y \in \R^n$ be unit $1$-eigenvectors of $G$ and $\Gc$, respectively. Then:
    \begin{itemize}
        \item[(a)] $\x \perp \y$.
        
        \item[(b)] $G$ does not have a cut vertex.
        
        \item[(c)] $2 \leq d(v) \leq n-3$.
        
        \item[(d)] The matrix $L(G) - 2I + \frac{2}{n}J + \x\x^T$ is positive semidefinite.
        
        \item[(e)] $\displaystyle d(u)x_u = -\sum_{\substack{w \in V(G) \\ w \notin N(u) \cup \{u\}}}x_w$ \ for every $u \in V(G)$.
    \end{itemize}
\end{lemma}

\begin{proof}
    Part~(a) follows from the fact that $\y$ corresponds to the eigenvalue $1$ of $L(\Gc)$ exactly when it corresponds to the eigenvalue $n-1$ of $L(G)$. Since $\x$ and $\y$ correspond to distinct eigenvalues of $L(G)$, they are perpendicular to each other.
    
    To prove part~(b), notice that if $G$ had a cut vertex, then its vertex connectivity and algebraic connectivity would both equal $1$. By \cite[Theorem~2.1]{KMNS02}, $G$ would then be a join of two nonempty graphs, contradicting the connectedness of $\Gc$. Now applying part~(b) to $\Gc$ shows that it also has no cut vertex. Since both graphs are connected and $n \geq 3$, neither can have a vertex of degree $1$, which establishes part~(c).

    To prove part~(d), choose a real orthonormal eigenbasis $\one/\sqrt{n},\x,\vv_3,\ldots,\vv_n$ of $L(G)$. Then
    \begin{align*}
        L(G)-2I+\frac{2}{n}J+\x\x^T = \sum_{j=3}^n(\lambda_j(G)-2)\vv_j\vv_j^T.
    \end{align*}
    Since $\lambda_j(G) \geq 2$ for all $j \geq 3$, this matrix is positive semidefinite.

    Finally, the fact that $\x \perp \one$ tells us that the entries of $\x$ sum to $0$. Combining this with the eigenvector equation (Equation~\eqref{eq:eigenvector_equation}) at $u$ gives
    \begin{align*}
        (d(u)-1)x_u = \sum_{w\in N(u)}x_w = -x_u-\sum_{\substack{w\in V(G)\\w\notin N(u)\cup\{u\}}}x_w.
    \end{align*}
    Adding $x_u$ to both sides proves part~(e).
\end{proof}

\subsection{Extreme eigenvector entries}\label{sec:vertices_extreme}

We can now prove the main result of this subsection, which is a result that tells us where the maximal and minimal entries of the (unique up to scalar multiplication) $1$-eigenvectors of $G$ and $\Gc$ are located. In particular, the following lemma shows that these four entries correspond to an induced path on four vertices of $G$ and of $\Gc$ that is made up of a dominating set of $G$ and a dominating set of $\Gc$:

\begin{lemma}\label{lem:extremal_path}
    Let $G$ be an eigensimple graph on $n \geq 4$ vertices, and let $\x,\y \in \R^n$ be unit $1$-eigenvectors of $L(G)$ and $L(\Gc)$, respectively. The sign of $\y$ can be chosen so that there exist distinct $p,q,r,s \in V(G)$ such that $p - q - r - s$ is an induced path of $G$ with
    \begin{align*}
        x_p = \max(\x), \quad y_q =\max(\y), \quad y_r = \min(\y), \ \ \text{and} \ \ x_s = \min(\x).
    \end{align*}
    Furthermore, $\{q,r\}$ dominates $G$ and $\{p,s\}$ dominates $\Gc$.
\end{lemma}

\begin{proof}
    We know from \cite{DLC10} that both $G$ and $\Gc$ have diameter (i.e., maximum distance between any pair of vertices) equal to $3$.

    We first show that a maximum and a minimum of each eigenvector can be chosen at distance exactly $3$ in its corresponding graph. Let $H \in \{G,\Gc\}$ and let $\z$ be a unit $1$-eigenvector of $H$. Since $\z \perp \one$ and $\z$ is nonzero, we have $\max(\z) > 0$ and $\min(\z) < 0$. Fix $s$ with $z_s = \min(\z)$, and let $M := \{v \in V(H) : z_v = \max(\z)\}$. For every $p \in M$, the eigenvector equation (Equation~\eqref{eq:eigenvector_equation}) gives
    \begin{align}\label{eq:eigenreduction_eigenvector}
        z_p-z_s = \sum_{v \in N_H(p)}(z_p-z_v)+\sum_{v \in N_H(s)}(z_v-z_s).
    \end{align}
    All summands on the right-hand side are nonnegative. If $p$ and $s$ were adjacent, their edge would contribute $z_p-z_s$ to each sum, already making the right-hand side larger than the left-hand side. We thus conclude that they are a distance of at least $2$ from each other.

    Suppose (for the sake of establishing a contradiction) that every vertex $p \in M$ is at distance exactly $2$ from $s$. Every common neighbor of $p$ and $s$ contributes $z_p-z_s$ to the right-hand side of Equation~\eqref{eq:eigenreduction_eigenvector}. There is therefore exactly one common neighbor, say $w$, and all other summands vanish. In particular, every other neighbor of $p$ belongs to $M$, and every other neighbor $v$ of $s$ attains $z_v = z_s = \min(\z) < 0$. The eigenvector equation at $p$ reduces to $z_p = z_p - z_w$, so $z_w = 0$.
    
    Since $w$ is the unique zero-valued neighbor of $s$, and $s$ does not depend on the choice of $p \in M$, $w$ does not depend on the choice of $p$ either. We have thus shown that every vertex of $M$ is adjacent to $w$, and there are no edges between $M$ and $V(H) \setminus (M\cup\{w\})$. Since the latter set contains $s$, deleting $w$ disconnects $H$, contradicting Lemma~\ref{lem:eigensimple_properties}(b). Thus some $p \in M$ has distance $3$ from $s$, proving the claim.

    Now choose $p$ and $s$ where $\x$ attains its maximum and minimum, respectively, at distance $3$ in $G$, and choose extrema $q$ and $r$ of $\y$ at distance $3$ in $\Gc$. Since $q$ and $r$ are not adjacent and have no common neighbor in $\Gc$, they are adjacent in $G$ and $\{q,r\}$ dominates $G$. The same argument shows that $\{p,s\}$ dominates $\Gc$. Furthermore, $p$, $q$, $r$, and $s$ are all distinct because, for example, the fact that $\{q,r\}$ dominates $G$ tells us that $q$ and $r$ are at a distance of at most $2$ from every vertex in $G$, so neither of them can equal $p$ or $s$ (which are a distance of $3$ from some vertex in $G$).

    Since $\{q,r\}$ dominates $G$, each of $p$ and $s$ is adjacent to at least one of $q$ or $r$. Since $p$ and $s$ are at distance $3$ from each other in $G$, they cannot have a common neighbor, so they must each be adjacent to \emph{exactly} one of $q$ or $r$. Interchanging $q$ and $r$, if necessary, results in $p-q-r-s$ being an induced path, and then choosing the sign of $\y$ so that $y_q = \max(\y)$ and $y_r = \min(\y)$ completes the proof.
\end{proof}

\subsection{Bounds on eigenvector coordinates}\label{sec:amplitudes}

We now show that the largest and smallest entries of a $1$-eigenvector cannot be too large relative to the other entries with the same sign.

\subsubsection{Bounding the extreme coordinates}\label{sec:relative_extreme}

We can now state our main bounds on the extreme entries. We first define a function that will be used in the proof of the bound and also in some later results.

\begin{definition}
    Define the function $h : \{2,3,4,\ldots\} \rightarrow (0,5/13]$ by
    \begin{align*}
        h(x) := \min\left\{\frac{5}{13}, \frac{x}{x^2-x+1}\right\}.
    \end{align*}
\end{definition}

The function $h$ is nonincreasing, with $h(2) = h(3) = 5/13$ and $h(d) = d/(d^2-d+1)$ for every integer $d \geq 4$.

\begin{lemma}\label{lem:max_513_bound}
    Let $G$ be an eigensimple graph on $n \geq 32$ vertices with $1$-eigenvector $\x \in \R^n$. Then
    \begin{align*}
        \frac{\max(\x)^2}{\|\x_+\|^2} \leq \frac{5}{13} \quad \text{and} \quad \frac{\min(\x)^2}{\|\x_-\|^2} \leq \frac{5}{13}.
    \end{align*}
\end{lemma}

\begin{proof}
    Since $\x \perp \one$ and $\x \neq \0$, both $\|\x_+\|$ and $\|\x_-\|$ are strictly positive (and in particular, the divisions in the statement of the lemma make sense). Replacing $\x$ by $-\x$ interchanges the two inequalities, so it suffices to prove the inequality $13\max(\x)^2 \leq 5\|\x_+\|^2$. Furthermore, by dividing $\x$ by $\|\x_+\|$ (i.e., assuming that $\|\x_+\| = 1$), it suffices to prove $13\max(\x)^2 \leq 5$.

    Suppose that $v \in V(G)$ is such that $x_v > 0$. The eigenvector equation (Equation~\eqref{eq:eigenvector_equation}) at $v$ implies that $\sum_{w \in N(v)} x_w = (d(v)-1)x_v$. Omitting the nonpositive terms from this sum gives
    \[
        \sum_{\substack{w \in N(v)\\x_w > 0}} x_w \geq \big(d(v)-1\big)x_v.
    \]
    The Cauchy--Schwarz inequality, applied to the vector of positive neighbor entries and the all-ones vector of the same dimension, then gives
    \begin{align}\label{eq:h_bound}
        \big(d(v)-1\big)^2 x_v^2 \leq d(v)\big(\|\x_+\|^2 - x_v^2\big) = d(v)\big(1 - x_v^2\big), \quad \text{so} \quad x_v^2 \leq \frac{d(v)}{d(v)^2 - d(v) + 1}.
    \end{align}
    In particular, $x_v^2 \leq h(d(v)) \leq 5/13$ when $d(v) \geq 4$. That is, if $\max(\x) = x_v$ for some vertex $v$ with $d(v) \geq 4$, then $\max(\x)^2 \leq 5/13$.

    By Lemma~\ref{lem:eigensimple_properties}(c), the only remaining possibilities are that the maximum occurs at a vertex of degree $2$ or $3$. These two cases are quite long and technical, so we defer their proof to Appendix~\ref{app:513_bound_complete}.
\end{proof}

\subsubsection{Bounding the range}\label{sec:range_bound}

For a unit $1$-eigenvector $\x$, the bounds from Lemma~\ref{lem:max_513_bound} on the largest and smallest entries of $\x$ can be combined because $\|\x_+\|^2 + \|\x_-\|^2 = 1$. It will be convenient for us to express these bounds via the following function:

\begin{definition}\label{defn:range_function_R}
    For $s \in (0,5/13]$ and $t \in [0,1/2]$, define
    \begin{align*}
        R(s,t) := \frac{t^2}{s}+\frac{13}{5}(1-t)^2.
    \end{align*}
\end{definition}

\begin{lemma}\label{lem:R_lower_bound}
    Let $R$ be as in Definition~\ref{defn:range_function_R}. Then $R(s,t) > 1$ for all $s \in (0,5/13]$ and $t \in [0,1/2]$.
\end{lemma}

\begin{proof}
    The Cauchy--Schwarz inequality applied to the vectors $\big(\sqrt{s}, \sqrt{5/13}\big)$ and $\big(t/\sqrt{s}, (1-t)\sqrt{13/5}\big)$ gives
    \[
        1 = \big(t + (1-t)\big)^2 \leq (s+5/13)R(s,t)
    \]
    for all $s \in (0,5/13]$ and $t \in [0,1/2]$. Since $s + 5/13 \leq 10/13 < 1$, this implies $R(s,t) > 1$ and proves the lemma.
\end{proof}

\begin{lemma}\label{lem:range_bound}
    Let $G$ be an eigensimple graph on $n \geq 32$ vertices, let $\x \in \R^n$ be a unit $1$-eigenvector of $G$, let $p,s \in V(G)$ satisfy $x_p = \max(\x)$ and $x_s = \min(\x)$, and suppose $2 \leq j \leq \max\{d(p),d(s)\}$ is an integer. Then
    \begin{align*}
         \range(\x)^2 \leq \frac{1}{R(h(j),\spread(\x))}.
    \end{align*}
\end{lemma}

\begin{proof}
    Since replacing $\x$ by $-\x$ and interchanging $p$ and $s$ preserves $\range(\x)$ and $\spread(\x)$, we can assume without loss of generality that $d(p)=\max\{d(p),d(s)\}$. Since $h$ is nonincreasing and $j\leq d(p)$, Inequality~\eqref{eq:h_bound} and Lemma~\ref{lem:max_513_bound} give
    \begin{align*}
        x_p^2 \leq h(d(p))\|\x_{+}\|^2 \leq h(j)\|\x_{+}\|^2 \qquad \text{and} \qquad x_s^2 \leq \frac{5}{13}\|\x_{-}\|^2.
    \end{align*}
    Dividing by $h(j)$ in the left inequality, by $5/13$ in the right inequality, and adding gives
    \begin{align}\label{ineq:p_h_135_bound}
        \frac{x_p^2}{h(j)} + \frac{13}{5}x_s^2 \leq \|\x_{+}\|^2 + \|\x_{-}\|^2 = 1.
    \end{align}

    We now develop a lower bound on the left-hand side in terms of $\range(\x)$ and $\spread(\x)$. Since $x_p > 0$, $x_s < 0$, and $\range(\x) = x_p-x_s$, we have
    \begin{align*}
        \range(\x)\spread(\x) = \min\{x_p,-x_s\} \qquad \text{and} \qquad \range(\x)\big(1-\spread(\x)\big) = \max\{x_p,-x_s\}.
    \end{align*}
    Squaring these identities and using the definition of $R$ (Definition~\ref{defn:range_function_R}) then gives
    \begin{align*}
        \range(\x)^2R\big(h(j),\spread(\x)\big) = \frac{\min\{x_p^2,x_s^2\}}{h(j)} + \frac{13}{5}\max\{x_p^2,x_s^2\}.
    \end{align*}

    If $x_p^2 \leq x_s^2$, this is exactly the quantity from Inequality~\eqref{ineq:p_h_135_bound} that we showed is no larger than $1$. If instead $x_p^2 > x_s^2$, then $h(j) \leq 5/13$ implies
    \begin{align*}
        \left(\frac{x_p^2}{h(j)}+\frac{13}{5}x_s^2\right) - \left(\frac{x_s^2}{h(j)}+\frac{13}{5}x_p^2\right) = \left(\frac{1}{h(j)}-\frac{13}{5}\right)(x_p^2-x_s^2)\geq0.
    \end{align*}
    We have thus shown in both cases that $\range(\x)^2 R\big(h(j),\spread(\x)\big) \leq 1$. Dividing both sides by $R(h(j),\spread(\x)) > 0$ proves the lemma.
\end{proof}

\subsection{Comparing the two eigenvectors}\label{sec:two_eigenvectors}

We now develop some bounds that let us compare a unit $1$-eigenvector $\x$ of an eigensimple graph $G$ to a unit $1$-eigenvector $\y$ of its complement.

\subsubsection{An identity involving all vertices}\label{sec:all_vertices_identity}

For a unit $1$-eigenvector $\x$ of $G$, we have
\begin{align*}
    \sum_{\{i,j\} \in E(G)}(x_i-x_j)^2 = \x^TL\x = 1.
\end{align*}
The same identity holds for a unit $1$-eigenvector $\y$ of $\Gc$. To relate these two eigenvectors to each other, we consider sums of products of the form $(x_i-x_j)^2(y_i-y_j)^2$. This idea comes from \cite{EK21}, where these products were used to compare the algebraic connectivity of a graph and its complement.

\begin{definition}\label{defn:quartic_functions}
    Let $G$ be an eigensimple graph on $n \geq 4$ vertices, and let $\x,\y \in \R^n$ be unit $1$-eigenvectors of $G$ and $\Gc$, respectively. Define
    \begin{align*}
        Q(G) & := \range(\x)^2 + \range(\y)^2 - \sum_{i < j}(x_i-x_j)^2(y_i-y_j)^2.
    \end{align*}
\end{definition}

In particular, notice that $Q(G)$ does not change if the signs of $\x$ and $\y$ are flipped, so it is independent of which unit $1$-eigenvectors $\x$ and $\y$ of $G$ and $\Gc$ are chosen (i.e., $Q(G)$ is well-defined). Furthermore, $Q(G) = Q(\Gc)$ and
\begin{align}\begin{split}\label{eq:QC_alternate}
    Q(G) & = \sum_{\{i,j\} \in E(G)}(x_i-x_j)^2\big(\range(\y)^2 - (y_i-y_j)^2\big) \\
    & \quad + \sum_{\{i,j\} \in E(\Gc)}\big(\range(\x)^2 - (x_i-x_j)^2\big)(y_i-y_j)^2,
\end{split}\end{align}
which shows that $Q(G) \geq 0$.

The following lemma provides our first connection between the $1$-eigenvectors $\x$ of $G$ and $\y$ of $\Gc$, by relating $Q(G)$ to the ranges of $\x$ and $\y$.

\begin{lemma}\label{lem:quartic_identities}
    Let $G$ be an eigensimple graph on $n \geq 4$ vertices, and let $\x,\y \in \R^n$ be unit $1$-eigenvectors of $G$ and $\Gc$, respectively. Then
    \begin{align*}
        1 = \range(\x)^2 + \range(\y)^2 - Q(G) - n\sum_{i \in V(G)}x_i^2y_i^2.
    \end{align*}
\end{lemma}

\begin{proof}
    The proof is by direct calculation. Since $\x,\y \perp \one$, $\|\x\| = \|\y\| = 1$, and $\x^T\y = 0$, we have
    \begin{align*}
        \sum_{i<j}(x_i-x_j)^2(y_i-y_j)^2 & = n\sum_i x_i^2y_i^2+\|\x\|^2\|\y\|^2+2(\x^T\y)^2 = 1+n\sum_{i \in V(G)} x_i^2y_i^2,
    \end{align*}
    where the leftmost sum picks some arbitrary ordering of the vertices in $V(G)$. It follows directly from the definition of $Q(G)$ that
    \begin{align*}
        Q(G) = \range(\x)^2 + \range(\y)^2 - 1 - n\sum_{i \in V(G)} x_i^2y_i^2,
    \end{align*}
    which is what we wanted to show.
\end{proof}

\subsubsection{An inequality from the extreme vertices}\label{sec:inequality_extreme_vertices}

We now turn the equality of Lemma~\ref{lem:quartic_identities} into an inequality that is based on just a few terms from the sum $\sum_{i \in V(G)}x_i^2y_i^2$ and the two sums from Equation~\eqref{eq:QC_alternate} (the alternate formula for $Q(G)$). In particular, notice that each term in these sums is nonnegative, so discarding some terms from these sums results in an inequality of the form $\range(\x)^2 + \range(\y)^2 - f(\x,\y) \geq 1$, where $f$ is some biquadratic function of $\x$ and $\y$.

We start by defining a helper function that we will need to make this idea precise:

\begin{definition}\label{defn:kappa_func}
    Define $\kappa : [0,1]^2 \rightarrow \R$ by
    \begin{align*}
        \kappa(x,y) := 1-4xy(1-x)(1-y) & - \frac{\big(x+(1-2x)y^2\big)^2}{1+30y^2} - \frac{\big(1-x+(2x-1)(1-y)^2\big)^2}{1+30(1-y)^2}\\
        & - \frac{\big(1-y+(2y-1)x^2\big)^2}{1+30x^2} - \frac{\big(y+(1-2y)(1-x)^2\big)^2}{1+30(1-x)^2}.
    \end{align*}
    Furthermore, define $\kappa_* : [0,1]^2 \rightarrow \R$ by $\kappa_*(x,y) = \min\{\kappa(x,y),\kappa(y,x)\}$.
\end{definition}

The functions $\kappa$ and $\kappa_*$ have the following elementary properties.

\begin{lemma}\label{lem:kappa_properties}
    Let $x,y \in [0,1]$. Then $\kappa(x,y) \leq 1$ and
    \begin{align*}
        \kappa(x,y) = \kappa(y,1-x) = \kappa(1-x,1-y) = \kappa(1-y,x).
    \end{align*}
    Consequently, $\kappa_*(x,y) \leq 1$ and
    \begin{align}\label{eq:kappa_symmetry}
        \kappa_*(x,y) = \kappa_*(y,x) = \kappa_*(1-x,y) = \kappa_*(x,1-y).
    \end{align}
\end{lemma}

\begin{proof}
    Every term subtracted from $1$ in the definition of $\kappa(x,y)$ (Definition~\ref{defn:kappa_func}) is nonnegative, so $\kappa(x,y) \leq 1$. The substitution $(x,y) \mapsto (y,1-x)$ leaves $xy(1-x)(1-y)$ unchanged and permutes the four rational terms, which proves the claimed identities for $\kappa$. The claims about $\kappa_*$ follow immediately.
\end{proof}

We can now state the inequality relating the two eigenvectors that we mentioned at the start of this subsection:

\begin{lemma}\label{lem:quartic_inequalities}
    Let $G$ be an eigensimple graph on $n \geq 32$ vertices and let $\x,\y \in \R^n$ be unit $1$-eigenvectors of $G$ and $\Gc$, respectively. Then
    \begin{align*}
        1 \leq \range(\x)^2 + \range(\y)^2 - \range(\x)^2\range(\y)^2\kappa_*\big(\spread(\x),\spread(\y)\big).
    \end{align*}
\end{lemma}

\begin{proof}
    Since $\range(\y) = \range(-\y)$ and $\spread(\y) = \spread(-\y)$, we can choose the sign of $\y$ and vertices $p,q,r,s$ as in Lemma~\ref{lem:extremal_path}. For ease of notation, write $r_x := \range(\x)^2$ and $r_y := \range(\y)^2$. By Lemma~\ref{lem:quartic_identities}, it suffices to prove that
    \begin{align*}
        Q(G) + n\sum_{i \in V(G)}x_i^2y_i^2 \geq r_xr_y\kappa_*\big(\spread(\x),\spread(\y)\big).
    \end{align*}

    At a maximum $x_v$ of $\x$, every summand $x_v - x_w$ on the right-hand side of the eigenvector equation (Equation~\eqref{eq:eigenvector_equation}) is nonnegative. It follows that a negative $x_w$ at a neighbor $w$ would make one summand larger than the left-hand side, contradicting the eigenvector equation. It follows that $x_w \geq 0$ for all neighbors $w$ of $v$. By applying this observation to each of $\x$, $-\x$, $\y$, and $-\y$, and using the adjacencies described by the induced path in Lemma~\ref{lem:extremal_path}, we get
    \begin{align}\label{ineq:coord_bounds}
        0 \leq x_q \leq x_p, \qquad x_s \leq x_r \leq 0, \qquad y_r \leq y_p \leq 0, \quad \text{and} \quad 0 \leq y_s \leq y_q.
    \end{align}

    We retain the terms of $Q(G)$ corresponding to the edges $\{p,q\}$ and $\{r,s\}$ of $G$ in the first sum of Equation~\eqref{eq:QC_alternate}, the terms corresponding to the edges $\{p,r\}$ and $\{q,s\}$ in the second sum, and the terms of $\sum_{i \in V(G)}x_i^2y_i^2$ corresponding to the vertices $p$, $q$, $r$, and $s$. All other terms are nonnegative, so we can discard them to get the inequality
    \begin{align}\begin{split}\label{ineq:QG_bound}
        n\left(\sum_{i \in V(G)}x_i^2y_i^2\right) + Q(G) & \geq n\big(y_q^2x_q^2+y_r^2x_r^2+x_p^2y_p^2+x_s^2y_s^2\big) \\[-0.4cm]
        & \quad + (x_p-x_q)^2\big(r_y-(y_q-y_p)^2\big)+(x_r-x_s)^2\big(r_y-(y_s-y_r)^2\big) \\
        & \quad + (y_p-y_r)^2\big(r_x-(x_p-x_r)^2\big)+(y_q-y_s)^2\big(r_x-(x_q-x_s)^2\big).
    \end{split}\end{align}

    We next bound this expression from below by a quadratic in $x_q$, $x_r$, $y_p$, and $y_s$. To bound the product $(x_p-x_q)^2\big(r_y-(y_q-y_p)^2\big)$, write
    \begin{align*}
        (x_p-x_q)^2 = x_p^2-(2x_px_q-x_q^2) \qquad \text{and} \qquad r_y-(y_q-y_p)^2 = r_y-y_q^2-(-2y_qy_p+y_p^2).
    \end{align*}
    The quantities $2x_px_q-x_q^2$ and $-2y_qy_p+y_p^2$ are nonnegative by Inequality~\eqref{ineq:coord_bounds}, so
    \begin{align*}
        (x_p-x_q)^2\big(r_y-(y_q-y_p)^2\big) & = \big(x_p^2-(2x_px_q-x_q^2)\big)\big(r_y-y_q^2-(-2y_qy_p+y_p^2)\big) \\
        & = x_p^2(r_y-y_q^2) - (2x_px_q-x_q^2)(r_y-y_q^2) - x_p^2(-2y_qy_p+y_p^2) \\
        & \quad \, + (2x_px_q-x_q^2)(-2y_qy_p+y_p^2) \\
        & \geq x_p^2(r_y-y_q^2) - (2x_px_q-x_q^2)(r_y-y_q^2) - x_p^2(-2y_qy_p+y_p^2).
    \end{align*}

    The other three products can be treated in exactly the same way. Substituting the resulting inequalities into Inequality~\eqref{ineq:QG_bound}, and then collecting terms, gives
    \begin{align}\begin{split}\label{inq:QG_better}
        n\left(\sum_{i \in V(G)}x_i^2y_i^2\right) + Q(G) & \geq r_xr_y-4x_px_sy_qy_r \\[-0.4cm]
        & \quad + \big(r_y+(n-2)y_q^2\big)x_q^2-2\big(x_pr_y-(x_p+x_s)y_q^2\big)x_q \\
        & \quad + \big(r_y+(n-2)y_r^2\big)x_r^2-2\big(x_sr_y-(x_p+x_s)y_r^2\big)x_r \\
        & \quad + \big(r_x+(n-2)x_p^2\big)y_p^2-2\big(y_rr_x-(y_q+y_r)x_p^2\big)y_p \\
        & \quad + \big(r_x+(n-2)x_s^2\big)y_s^2-2\big(y_qr_x-(y_q+y_r)x_s^2\big)y_s.
    \end{split}\end{align}

    We want to further lower bound this quantity, which we will do by finding a lower bound on each of the final four lines of Inequality~\eqref{inq:QG_better}. Applying the inequality $az^2 - 2bz \geq -b^2/a$ for $a > 0$ (since $a^2z^2 - 2abz + b^2 = (az - b)^2 \geq 0$) to each of these four lines gives
    \begin{align*}
        n\left(\sum_{i \in V(G)}x_i^2y_i^2\right) + Q(G) & \geq r_xr_y-4x_px_sy_qy_r \\[-0.4cm]
        & \quad -\frac{\big(x_pr_y-(x_p+x_s)y_q^2\big)^2}{r_y+(n-2)y_q^2}
        -\frac{\big(x_sr_y-(x_p+x_s)y_r^2\big)^2}{r_y+(n-2)y_r^2} \\
        & \quad -\frac{\big(y_rr_x-(y_q+y_r)x_p^2\big)^2}{r_x+(n-2)x_p^2}
        -\frac{\big(y_qr_x-(y_q+y_r)x_s^2\big)^2}{r_x+(n-2)x_s^2} \\
        & \geq r_xr_y\kappa\left(\frac{x_p}{x_p-x_s},\frac{y_q}{y_q-y_r}\right) \\
        & \geq r_xr_y\kappa_{*}\left(\frac{x_p}{x_p-x_s},\frac{y_q}{y_q-y_r}\right).
    \end{align*}
    Here the second inequality follows from $n \geq 32$: replacing each denominator of the form $a + (n-2)b^2$ by $a+30b^2$ can only increase the corresponding subtracted fraction. By Equation~\eqref{eq:kappa_symmetry},
    \begin{align*}
        \kappa_*\left(\frac{x_p}{x_p-x_s},\frac{y_q}{y_q-y_r}\right) = \kappa_*\left(\frac{\min\{x_p,-x_s\}}{x_p-x_s},\frac{\min\{y_q,-y_r\}}{y_q-y_r}\right) = \kappa_*\big(\spread(\x),\spread(\y)\big),
    \end{align*}
    which completes the proof.
\end{proof}

\subsubsection{Inequalities from dominating sets}\label{sec:complicated_bounds}

The inequality of Lemma~\ref{lem:quartic_inequalities} was obtained from the identity of Lemma~\ref{lem:quartic_identities} by discarding all except four terms from $\sum_{i\in V(G)}x_i^2y_i^2$ and four terms from the alternate formula for $Q(G)$. We now obtain two further bounds by instead discarding all terms of $Q(G)$ except for those corresponding to pairs of vertices that are incident with one of the dominating sets described by Lemma~\ref{lem:extremal_path}.

\begin{definition}\label{defn:gamma_Tj}
    Define $\gamma_3 := 347/392$, $\gamma_4 := 613/729$, and $\gamma_6 := 463/625$. For $j \in \{3,4,6\}$, define functions $T_{*}, T_j : [0,1/2] \rightarrow \R$ by
    \begin{align*}
        T_{*}(t) & := \frac{1}{8}+(1-t)^2-\frac{1}{13}\big(t^2+(1-t)^2\big) \quad \text{and} \\
        T_j(t) & := 1 - \gamma_j + \frac{40}{31}\gamma_j(1-t)^2 - \frac{10j}{(31-j)^2}\big(t^2+(1-t)^2\big).
    \end{align*}
\end{definition}

The reason for defining these seemingly bizarre quantities $\gamma_3$, $\gamma_4$ and $\gamma_6$, and functions $T_*$, $T_3$, $T_4$, and $T_6$, is that we will shortly have to split into a few different cases depending on the degrees of certain vertices in $G$ and $\Gc$. Degrees $3$, $4$, and $6$ are where we need to split our casework (i.e., one case will cover degrees $2$ and $3$, one will cover degree $4$, one will cover degrees $5$ and $6$, and one will cover degrees $7$ and larger). A different $\gamma_j$ and $T_j$ will be used in the inequalities of the different cases.

Since $(1-t)^2 \geq (t^2+(1-t)^2)/2$ for $t \in [0,1/2]$, we have
\begin{align}\label{eq:Tpositive}
    T_{*}(t) \geq \frac{1}{8}+\frac{11}{26}\big(t^2+(1-t)^2\big) > 0 \quad \text{and} \quad T_j(t) \geq 1 - \gamma_j + \frac{296}{775}\big(t^2+(1-t)^2\big) > 0
\end{align}
on that interval. In particular, this shows that the coefficients of the squared ranges in the following lemma are positive:

\begin{lemma}\label{lem:uniform_domination}
    Let $G$ be an eigensimple graph on $n \geq 32$ vertices, let $\x,\y \in \R^n$ be unit $1$-eigenvectors of $G$ and $\Gc$, respectively, choose $p,q,r,s$ as in Lemma~\ref{lem:extremal_path}, and let $j_x,j_y \in \{3,4,6\}$.
    
    \begin{itemize}
        \item[(a)] If $\max\big\{\overline{d}(q), \overline{d}(r)\big\} \leq 3$ then
        \begin{align*}
            1 \leq \range(\x)^2+T_*\big(\spread(\y)\big)\range(\y)^2+\frac{1}{13}.
        \end{align*}

        \item[(b)] If $\max\{d(p), d(s)\} \leq j_x$ and $\max\big\{\overline{d}(q), \overline{d}(r)\big\} \leq j_y$ then
        \begin{align*}
            1 & \leq T_{j_x}\big(\spread(\x)\big)\range(\x)^2+T_{j_y}\big(\spread(\y)\big)\range(\y)^2 + \frac{10j_x}{(31-j_x)^2}+\frac{10j_y}{(31-j_y)^2}.
        \end{align*}
    \end{itemize}
\end{lemma}

\begin{proof}
    Fix $j \in \{3,4,6\}$ with $\max\big\{\overline d(q), \overline d(r)\big\} \leq j$. Lemma~\ref{lem:eigensimple_properties}(e) (applied to the vertices $q$ and $r$) and the Cauchy--Schwarz inequality give
    \begin{align}\label{ineq:T_proof_dq_dr}
        d(q)^2 x_q^2 \leq j\sum_{w \in \overline{N}(q)}x_w^2 \qquad \text{and} \qquad d(r)^2 x_r^2 \leq j\sum_{w \in \overline{N}(r)}x_w^2.
    \end{align}
    The sets $\overline{N}(q)$ and $\overline{N}(r)$ are disjoint, since $\{q,r\}$ dominates $G$. Since $d(q), d(r) \geq n-1-j \geq 31-j$ and $\|\x\| = 1$, this disjointness gives (after dividing Inequalities~\eqref{ineq:T_proof_dq_dr} by $d(q)$, $d(r)$, $d(q)^2$, and/or $d(r)^2$ and then adding them) 
    \begin{align}\label{ineq:x_q_x_r_bound}
        x_q^2+x_r^2 \leq \frac{j}{(31-j)^2} \qquad \text{and} \qquad d(q)x_q^2+d(r)x_r^2 \leq \frac{j}{31-j}.
    \end{align}
    
    Now define
    \begin{align*}
        \Gamma_{q,r}(G) & := \sum_{v\in\{q,r\}} \left( \sum_{w\in N(v)\setminus\{q,r\}}(x_w - x_v)^2\right) \quad \text{and} \\
        Q_{q,r}(G) & := \Gamma_{q,r}(G)\range(\y)^2 - \sum_{v\in\{q,r\}} \left( \sum_{w\in N(v)\setminus\{q,r\}}(x_w - x_v)^2(y_w - y_v)^2 \right).
    \end{align*}
    The alternate formula
    \begin{align}\label{eq:Qqr_alternate}
        Q_{q,r}(G) & = \sum_{v\in\{q,r\}} \left( \sum_{w\in N(v)\setminus\{q,r\}}(x_w - x_v)^2\big(\range(\y)^2 - (y_w - y_v)^2\big) \right)
    \end{align}
    shows that $0 \leq Q_{q,r}(G) \leq Q(G)$. Furthermore, the eigenvector equation (Equation~\eqref{eq:eigenvector_equation}) at vertices $q$ and $r$ gives
    \[
        \sum_{w\in N(q)\setminus\{q,r\}} x_w = \big( d(q) - 1 \big)x_q - x_r \quad \text{and} \quad \sum_{w\in N(r)\setminus\{q,r\}} x_w = \big( d(r) - 1 \big)x_r - x_q.
    \]
    Expanding the squares in the definition of $\Gamma_{q,r}(G)$ and using these identities gives
    \[
        \Gamma_{q,r}(G) = \sum_{v\in\{q,r\}} \left( \sum_{w\in N(v)\setminus\{q,r\}}x_w^2\right) - \big( d(q) - 1 \big)x_q^2 - \big( d(r) - 1 \big)x_r^2 + 4x_qx_r.
    \]
    Since every vertex outside $\{q,r\}$ occurs in this double sum (since $\{q,r\}$ dominates $G$) and $\|\x\| = 1$, we get
    \begin{align}\begin{split}\label{ineq:gamma_Gamma}
        \Gamma_{q,r}(G) & \geq 1 - d(q)x_q^2 - d(r)x_r^2 + 4x_qx_r \\
        & \geq 1 - d(q)x_q^2 - d(r)x_r^2 - 2(x_q^2 + x_r^2) \\
        & \geq 1 - \frac{j}{31-j} - \frac{2j}{(31-j)^2} \\
        & = \gamma_j.
    \end{split}\end{align}
    (The final equality is why we defined $\gamma_3$, $\gamma_4$, and $\gamma_6$ the way that we did.)

    Our next goal is to derive a similar bound for $Q_{q,r}(G)$. For real numbers $\alpha, \beta$, and $\varepsilon$ with $\varepsilon > 0$, Young's inequality \cite[Section~4.8]{HLP52} says that $|\alpha\beta| \leq \alpha^2/(2\varepsilon) + \varepsilon\beta^2/2$. Using $\alpha = y_v$ and $\beta = y_w$ (and leaving $\varepsilon > 0$ unspecified) gives
    \[
        (y_w - y_v)^2 = y_v^2 - 2y_vy_w + y_w^2 \leq \left(1 + \frac{1}{\varepsilon}\right)y_v^2 + (1+\varepsilon)y_w^2,
    \]
    and similarly using $\alpha = x_v$, $\beta = x_w$, and $\varepsilon = 4/5$ gives
    \[
        (x_w - x_v)^2 = x_v^2 - 2x_vx_w + x_w^2 \leq \frac{9x_v^2}{4} + \frac{9x_w^2}{5}.
    \]
    Multiplying these inequalities together gives an upper bound on $(x_w-x_v)^2(y_w-y_v)^2$ and thus (via its definition) a lower bound on $Q_{q,r}(G)$. We obtain a sharper bound by applying the second inequality only to the term containing $y_w^2$ (not $y_v^2$):
    \begin{align*}
        (x_w-x_v)^2(y_w-y_v)^2 & \leq \left(1+\frac{1}{\varepsilon}\right)(x_w-x_v)^2y_v^2 + (1+\varepsilon)(x_w-x_v)^2 y_w^2 \\
        & \leq \frac{9(1+\varepsilon)}{5}x_w^2y_w^2 + \frac{9(1+\varepsilon)}{4}x_v^2y_w^2 + \left(1+\frac{1}{\varepsilon}\right)(x_w-x_v)^2y_v^2.
    \end{align*}
    (We emphasize that this inequality holds for all $\varepsilon > 0$; we will make two separate choices of $\varepsilon$ later.)
    
    We now sum this inequality over $v \in \{q,r\}$ and $w \in N(v)\setminus\{q,r\}$:
    \begin{align}\begin{split}\label{ineq:long_gamma}
        \Gamma_{q,r}(G)\range(\y)^2-Q_{q,r}(G) & = \sum_{v\in\{q,r\}}\left(\sum_{w\in N(v)\setminus\{q,r\}}(x_w-x_v)^2(y_w-y_v)^2\right) \\
        & \leq \frac{9(1+\varepsilon)}{5}\sum_{v\in\{q,r\}}\left(\sum_{w\in N(v)\setminus\{q,r\}}x_w^2 y_w^2\right) \\
        & \quad + \frac{9(1+\varepsilon)}{4}\sum_{v\in\{q,r\}}x_v^2\left(\sum_{w\in N(v)\setminus\{q,r\}}y_w^2\right) \\
        & \quad + \left(1+\frac{1}{\varepsilon}\right)\sum_{v\in\{q,r\}}\left(\sum_{w\in N(v)\setminus\{q,r\}}(x_w-x_v)^2y_v^2\right).
    \end{split}\end{align}
    
    We now bound the three double sums on the right-hand side of this inequality. In the first double sum, each term occurs at most twice (once for $v = q$ and once for $v = r$), so
    \[
        \sum_{v\in\{q,r\}}\left(\sum_{w\in N(v)\setminus\{q,r\}}x_w^2 y_w^2\right) \leq 2\left(\sum_{w\in V(G)}x_w^2 y_w^2\right).
    \]
    For the second double sum, we use the fact that $\|\y\| = 1$ to compute
    \begin{align*}
        \sum_{v\in\{q,r\}}x_v^2\left(\sum_{w\in N(v)\setminus\{q,r\}}y_w^2\right) \leq (x_q^2+x_r^2)(1-y_q^2-y_r^2) \leq \frac{j}{(31-j)^2}(1-y_q^2-y_r^2),
    \end{align*}
    where the final inequality comes from Inequality~\eqref{ineq:x_q_x_r_bound}. Finally, for the third double sum, we use the inequality $y_v^2 \leq \max\big\{y_q^2,y_r^2\big\}$, which gives
    \begin{align*}
        \sum_{v\in\{q,r\}}\left(\sum_{w\in N(v)\setminus\{q,r\}}(x_w-x_v)^2y_v^2\right) & \leq \max\big\{y_q^2,y_r^2\big\} \sum_{v\in\{q,r\}} \left( \sum_{w\in N(v)\setminus\{q,r\}}(x_w - x_v)^2\right) \\
        & = \max\big\{y_q^2,y_r^2\big\}\Gamma_{q,r}(G).
    \end{align*}
    
    Plugging these three bounds on the double sums from Inequality~\eqref{ineq:long_gamma} back into that inequality, and then rearranging, gives
    \begin{align*}
        & Q_{q,r}(G) + \frac{18(1+\varepsilon)}{5}\left(\sum_{w\in V(G)}x_w^2y_w^2\right) \\
        & \qquad \geq \Gamma_{q,r}(G)\left(\range(\y)^2 - \left(1+\frac{1}{\varepsilon}\right) \max\big\{y_q^2,y_r^2\big\}\right) - \frac{9(1+\varepsilon)j}{4(31-j)^2}\big(1 - y_q^2 - y_r^2\big).
    \end{align*}
    This in turn implies that, for all $0 < \delta \leq \gamma_j$, we have
    \begin{align}\begin{split}\label{eq:local_Q_bound}
        & Q_{q,r}(G) + \frac{18(1+\varepsilon)}{5}\left(\sum_{w\in V(G)}x_w^2y_w^2\right) \\
        & \qquad \geq \delta\left(\range(\y)^2 - \left(1+\frac{1}{\varepsilon}\right) \max\big\{y_q^2,y_r^2\big\}\right) - \frac{9(1+\varepsilon)j}{4(31-j)^2}\big(1 - y_q^2 - y_r^2\big).
    \end{split}\end{align}
    Indeed, if $\range(\y)^2 \geq (1+1/\varepsilon)\max\big\{y_q^2,y_r^2\big\}$ then we can use the fact that $\Gamma_{q,r}(G) \geq \gamma_j \geq \delta$ from Inequality~\eqref{ineq:gamma_Gamma}. Otherwise, the right-hand side of Inequality~\eqref{eq:local_Q_bound} is nonpositive (recall that $1-y_q^2-y_r^2 \geq 1 - \|\y\|^2 = 0$), so the inequality is trivially true.

    We are finally in a position to prove part~(a) of the lemma. Let $j = 3$, $\delta = 7/8$, and $\varepsilon = 7$ in Inequality~\eqref{eq:local_Q_bound} (notice that $7/8 \leq \gamma_3$, so this choice of $\delta$ is allowed). Since $n \geq 144/5 = 18(1+\varepsilon)/5$, we get
    \begin{align*}
        Q_{q,r}(G) + n\left(\sum_{w\in V(G)}x_w^2y_w^2\right) & \geq \frac{7}{8}\range(\y)^2 - \max\big\{y_q^2,y_r^2\big\} - \frac{27}{392}\big(1 - y_q^2 - y_r^2\big).
    \end{align*}
    By making use of the facts that
    \begin{align*}
        \max\big\{y_q^2,y_r^2\big\} = \big(1-\spread(\y)\big)^2\range(\y)^2 \quad \text{and} \quad y_q^2 + y_r^2 = \big(\spread(\y)^2 + \big(1-\spread(\y)\big)^2\big)\range(\y)^2,
    \end{align*}
    this becomes
    \begin{align*}
        & Q_{q,r}(G) + n\left(\sum_{w\in V(G)}x_w^2y_w^2\right) \\
        & \qquad \geq \frac{7}{8}\range(\y)^2 - \max\big\{y_q^2,y_r^2\big\} - \frac{27}{392}\big(1 - y_q^2 - y_r^2\big) \\
        & \qquad = \frac{7}{8}\range(\y)^2 - \big(1-\spread(\y)\big)^2\range(\y)^2 + \frac{27}{392}\big(\spread(\y)^2 + \big(1-\spread(\y)\big)^2\big)\range(\y)^2 - \frac{27}{392} \\
        & \qquad =\left( \frac{7}{8} + \frac{27}{392}\spread(\y)^2 - \frac{365}{392}\big(1-\spread(\y)\big)^2 \right)\range(\y)^2 - \frac{27}{392} \\
        & \qquad \geq \Big(1 - T_{*}\big(\spread(\y)\big)\Big)\range(\y)^2 - \frac{1}{13}.
    \end{align*}
    (We note that we replaced $27/392$ by the larger $1/13$, and we similarly adjusted coefficients when defining $T_{*}$ back in Definition~\ref{defn:gamma_Tj}, just to make this final inequality slightly cleaner.)
    
    Since $Q_{q,r}(G) \leq Q(G)$, Lemma~\ref{lem:quartic_identities} now gives
    \begin{align*}
        1 & = \range(\x)^2 + \range(\y)^2 - Q(G) - n\left(\sum_{w\in V(G)}x_w^2y_w^2\right) \\
        & \leq \range(\x)^2 + T_{*}\big(\spread(\y)\big)\range(\y)^2 + \frac{1}{13},
    \end{align*}
    which proves part~(a).
    
    For part~(b), instead let $j = j_y$, $\delta = \gamma_{j_y}$, and $\varepsilon = 31/9$ in Inequality~\eqref{eq:local_Q_bound}. Since $n/2 \geq 16 = 18(1+\varepsilon)/5$, a similar calculation to the one performed in part~(a) gives
    \begin{align}\label{ineq:Qqr_partb_1}
        Q_{q,r}(G) + \frac{n}{2}\left(\sum_{w\in V(G)}x_w^2y_w^2\right) \geq \Big(1 - T_{j_y}\big(\spread(\y)\big)\Big)\range(\y)^2 - \frac{10j_y}{(31-j_y)^2}.
    \end{align}
    Now apply the same bound in $\Gc$, with eigenvectors $(\y,-\x)$, ordered vertices $(q,s,p,r)$ (instead of $(p,q,r,s)$), and $j = j_x$, so that
    \begin{align}\label{ineq:Qqr_partb_2}
        Q_{p,s}(\Gc) + \frac{n}{2}\left(\sum_{w\in V(G)}x_w^2y_w^2\right) \geq \Big(1 - T_{j_x}\big(\spread(\x)\big)\Big)\range(\x)^2 - \frac{10j_x}{(31-j_x)^2}.
    \end{align}
    The terms in the alternate formula for $Q_{q,r}(G)$ (Equation~\eqref{eq:Qqr_alternate}) come from the first sum in Equation~\eqref{eq:QC_alternate}, while those of $Q_{p,s}(\Gc)$ come from the second. All other terms in Equation~\eqref{eq:QC_alternate} are nonnegative, so $Q_{q,r}(G) + Q_{p,s}(\Gc) \leq Q(G)$. It follows that if we add Inequalities~\eqref{ineq:Qqr_partb_1} and~\eqref{ineq:Qqr_partb_2} and apply Lemma~\ref{lem:quartic_identities} then we get
    \begin{align*}
        1 & = \range(\x)^2 + \range(\y)^2 - Q(G) - n\left(\sum_{w\in V(G)}x_w^2y_w^2\right) \\
        & \leq T_{j_x}\big(\spread(\x)\big)\range(\x)^2 + T_{j_y}\big(\spread(\y)\big)\range(\y)^2 + \frac{10j_x}{(31-j_x)^2} + \frac{10j_y}{(31-j_y)^2},
    \end{align*}
    proving part~(b).
\end{proof}

\subsection{Putting it all together}\label{sec:finishing}

We now combine our eigenvector bounds to rule out the existence of eigensimple graphs on $n \geq 32$ vertices.

\subsubsection{Eleven inequalities}\label{sec:eleven_inequalities}

The final ingredient that we need is a family of functions that are built out of the same ingredients as the various inequalities that we proved over the course of the past few subsections. Recall the definitions of $R$ from Definition~\ref{defn:range_function_R}, $\kappa_*$ from Definition~\ref{defn:kappa_func}, and $T_{*}$ and $T_j$ from Definition~\ref{defn:gamma_Tj}.

\begin{definition}\label{defn:exclusion_functions}
    Define $\delta_3 := 5/13$, $\delta_4 := 4/13$, and $\delta_6 := 5/21$. For $i,j \in \{3,4,6\}$ and $x,y \in [0,1/2]$, define
    \begin{align}
        F_{\textup{A}}(x,y) & := R\left(\frac{7}{43},x\right) R\left(\frac{4}{13},y\right) - R\left(\frac{7}{43},x\right) - R\left(\frac{4}{13},y\right) + \kappa_*(x,y),\notag\\
        F_{\textup{B}}(x,y) & := \frac{25}{13}R\left(\frac{7}{43},x\right)R\left(\frac{5}{13},y\right) - 2R\left(\frac{5}{13},y\right) - \big(1+T_{*}(y)\big)R\left(\frac{7}{43},x\right) + \kappa_{*}(x,y),\notag\\
        F_{i,j}(x,y) & := \left(2 - \frac{10i}{(31-i)^2} - \frac{10j}{(31-j)^2}\right)R(\delta_i,x)R(\delta_j,y)\notag\\
        & \quad - \big(1+T_{i}(x)\big)R(\delta_j,y) - \big(1+T_j(y)\big)R(\delta_i,x) + \kappa_{*}(x,y).\label{eq:Fij}
    \end{align}
\end{definition}

\begin{lemma}\label{lem:eleven_inequalities}
    Let $i,j \in \{3,4,6\}$. Then
    \begin{align*}
        F_{\textup{A}}(x,y) > 0, \qquad F_{\textup{B}}(x,y) > 0, \qquad F_{i,j}(x,y) > 0
    \end{align*}
    for all $x,y \in [0,1/2]$.
\end{lemma}

We defer the proof of this lemma to Appendix~\ref{app:certificates}. To prove our main result (Theorem~\ref{thm:no_large_eigensimple}), we will show in the next subsection that if an eigensimple graph on $n \geq 32$ vertices exists then at least one of these functions must become nonpositive at some point, contradicting the above lemma and proving nonexistence of such graphs.

\subsubsection{Completion of the proof}\label{sec:finish}

We are now in a position to put the various pieces of the proof of Theorem~\ref{thm:no_large_eigensimple} together and thus establish that there do not exist eigensimple graphs on $n \geq 32$ vertices.

\begin{proof}[Proof of Theorem~\ref{thm:no_large_eigensimple}]
    Suppose (in order to establish a contradiction) that $G$ is an eigensimple graph on $n \geq 32$ vertices. Choose unit $1$-eigenvectors $\x$ and $\y$ of $G$ and $\Gc$, respectively, and choose $p,q,r,s$ as in Lemma~\ref{lem:extremal_path}. For simplicity of notation, define
    \begin{align*}
        d_x & := \max\{d(p),d(s)\}, & r_x & := \range(\x)^2, & s_x & := \spread(\x), \\
        d_y & := \max\{\overline{d}(q),\overline{d}(r)\}, & r_y & := \range(\y)^2, & s_y & := \spread(\y).
    \end{align*}
    We have $s_x, s_y \in (0,1/2]$. Lemma~\ref{lem:quartic_inequalities} gives
    \begin{align}\label{eq:rx_inequality}
        1 \leq r_x + r_y - r_xr_y\kappa_*(s_x,s_y),
    \end{align}
    and Lemma~\ref{lem:range_bound} gives
    \begin{align}\label{ineq:rx_ry_bounds}
        r_x \leq \frac{1}{R(h(j_x),s_x)} \quad \text{and} \quad r_y \leq \frac{1}{R(h(j_y),s_y)}
    \end{align}
    whenever $j_x \in \{2,3,\ldots,d_x\}$ and $j_y \in \{2,3,\ldots,d_y\}$.
    
    Notice that the right-hand side of Inequality~\eqref{eq:rx_inequality} (i.e., $r_x + r_y - r_xr_y\kappa_*(s_x,s_y)$) is nondecreasing in both $r_x$ and $r_y$ since, for example, its partial derivative with respect to $r_x$ is $1 - r_y\kappa_*(s_x,s_y) > 0$, with this final inequality following from the facts that $r_y < 1$ (by Lemmas~\ref{lem:R_lower_bound} and~\ref{lem:range_bound}) and $\kappa_*(s_x,s_y) \leq 1$ (by Lemma~\ref{lem:kappa_properties}). The upper bounds in Inequality~\eqref{ineq:rx_ry_bounds} are also less than $1$ by Lemma~\ref{lem:R_lower_bound}, since $h(j_x),h(j_y) \leq 5/13$. It follows (by plugging those upper bounds into Inequality~\eqref{eq:rx_inequality}) that
    \begin{align}\label{eq:rx_inequality_bounds}
        1 \leq \frac{1}{R(h(j_x),s_x)} + \frac{1}{R(h(j_y),s_y)} - \frac{1}{R(h(j_x),s_x)R(h(j_y),s_y)}\kappa_*(s_x,s_y).
    \end{align}
    We now split into two cases depending on the values of $d_x$ and $d_y$:\medskip
    
    \noindent\underline{\textbf{Case 1:} $\max\{d_x,d_y\} \geq 7$.} Since $\kappa_*(s_x,s_y) = \kappa_*(s_y,s_x)$, we can interchange $G$ and $\Gc$ and thus assume without loss of generality that $d_x \geq 7$. If $d_y \geq 4$, taking $j_x = 7$ and $j_y = 4$ and then multiplying both sides of Inequality~\eqref{eq:rx_inequality_bounds} by
    \[
        R\big(h(j_x),s_x\big)R\big(h(j_y),s_y\big) = R\left(\frac{7}{43},s_x\right)R\left(\frac{4}{13},s_y\right) > 0
    \]
    gives exactly $F_{\textup{A}}(s_x,s_y) \leq 0$, which contradicts Lemma~\ref{lem:eleven_inequalities}.

    If instead $d_y \leq 3$, then Lemma~\ref{lem:uniform_domination}(a) tells us that
    \begin{align}\label{eq:single_y_again}
        1 \leq r_x + T_{*}(s_y)r_y + \frac{1}{13}.
    \end{align}
    Since $T_{*}(s_y) > 0$ by Inequality~\eqref{eq:Tpositive}, we can substitute the upper bounds from Inequality~\eqref{ineq:rx_ry_bounds} with $j_x = 7$ and $j_y = 2$ into Inequality~\eqref{eq:single_y_again}. Adding this bound to Inequality~\eqref{eq:rx_inequality_bounds} with the same thresholds, and then multiplying both sides by
    \[
        R\big(h(j_x),s_x\big)R\big(h(j_y),s_y\big) = R\left(\frac{7}{43},s_x\right)R\left(\frac{5}{13},s_y\right) > 0,
    \]
    gives the inequality $F_{\textup{B}}(s_x,s_y) \leq 0$. This again contradicts Lemma~\ref{lem:eleven_inequalities} and shows that this case cannot happen.\medskip
    
    \noindent\underline{\textbf{Case 2:} $d_x,d_y \leq 6$.} By Lemma~\ref{lem:eigensimple_properties}(c), we have $d_x, d_y \geq 2$. Let $i$ and $j$ be the smallest elements of $\{3,4,6\}$ that are at least $d_x$ and $d_y$, respectively (e.g., if $(d_x,d_y) = (4,5)$ then $(i,j) = (4,6)$). Choose $j_x$ to be $2$, $4$, or $5$ according as $i$ is $3$, $4$, or $6$, and choose $j_y$ in the same way using $j$. Then $h(j_x) = \delta_i$ and $h(j_y) = \delta_j$, while $d_x \leq i$ and $d_y \leq j$. Applying Lemma~\ref{lem:uniform_domination}(b) with its degree bounds equal to $i$ and $j$ gives
    \[
        1 \leq T_{i}(s_x)r_x + T_{j}(s_y)r_y + \frac{10i}{(31-i)^2} + \frac{10j}{(31-j)^2}.
    \]
    Since $T_{i}(s_x) > 0$ and $T_{j}(s_y) > 0$, we can plug in the upper bounds from Inequality~\eqref{ineq:rx_ry_bounds}. Adding the resulting inequality to Inequality~\eqref{eq:rx_inequality_bounds}, and then multiplying both sides by $R(\delta_i,s_x)R(\delta_j,s_y) > 0$, gives the inequality $F_{i,j}(s_x,s_y) \leq 0$. This again contradicts Lemma~\ref{lem:eleven_inequalities} and completes the proof.
\end{proof}

With Theorem~\ref{thm:no_large_eigensimple} proved, we immediately have the following corollary:

\begin{corollary}\label{cor:snn_large_n}
    Suppose $n \geq 32$ is an integer. Then there does not exist a simple graph on $n$ vertices with Laplacian spectrum $S_{n,n}$.
\end{corollary}

When combined with Theorem~\ref{thm:certified_small_orders}, this proves Conjecture~\ref{conj:Snn}:

\begin{theorem}\label{thm:Snn}
    Suppose $n \geq 2$ is an integer. Then there does not exist a simple graph on $n$ vertices with Laplacian spectrum $S_{n,n}$.
\end{theorem}

\section*{Acknowledgements}

This paper was developed with a significant amount of help from, and dozens of conversations with, ChatGPT versions 5.5, 5.6 Sol, and 6 Astra. The accompanying code \cite{JohCertificates} was originally written by ChatGPT-6 Astra. The author has checked, simplified, and clarified all proofs, has verified all computations, and takes full responsibility for their correctness. The author acknowledges support from NSERC Discovery Grant number RGPIN-2022-04098.

\bibliographystyle{alpha}
\bibliography{references}

\appendix

\section{Bounds at vertices of degree 2 or 3}\label{app:513_bound_complete}

In this appendix, we complete the proof of Lemma~\ref{lem:max_513_bound}. We already proved that lemma in the case when $\max(\x)$ is attained at a vertex of degree at least $4$, so all that remains is to prove the degree $2$ and $3$ cases. We start with a lemma about grounded Laplacians \cite{Mie93}, which are matrices of the form $L(G)+\diag(\d)$, where $\d$ is a nonnegative real vector and every connected component of $G$ contains a vertex $v$ with $d_v > 0$.

\begin{lemma}\label{lem:grounded}
    Let $G$ be a simple graph on $n\geq1$ vertices, let $\d\in\mathbb Z_{\geq0}^n$ have exactly $k$ entries equal to $0$, and suppose that $K:=L(G)+\diag(\d)$ is a grounded Laplacian. Then the following hold, where $\z:=K^{-1}\one$ in parts~(c) and~(d):
    \begin{itemize}
        \item[(a)] $K$ is positive definite (and thus invertible).

        \item[(b)] $K^{-1}\geq0$ entrywise.

        \item[(c)] $\displaystyle \sum_{v\in V(G)}z_v\leq n+\frac{k(k+1)(2k+7)}{6}$.

        \item[(d)] $\displaystyle \sum_{v\in W}z_v\leq(k+1)|W|+\frac{k(k+1)(2k+1)}{6}$ \ for every $W\subseteq V(G)$.
    \end{itemize}
\end{lemma}

\begin{proof}
    Parts~(a) and~(b) are standard (see \cite{DB13}, for example), and they imply $\z \geq \0$. We prove part~(d) next. For any $S\subseteq V(G)$, summing the equations $K\z = \one$ over $S$ gives
    \begin{align}\label{eq:grounded_laplac_sum}
        \sum_{v\in S} d_vz_v + \sum_{\substack{v\in S,\ w\notin S\\\{v,w\}\in E(G)}} (z_v - z_w) = |S|,
    \end{align}
    since the contributions from edges with both $v$ and $w$ in $S$ cancel.

    Define $Z := \{v\in V(G) : d_v=0\}$ and $M := \max_{v \notin Z}z_v$. Choose $m \in V(G) \setminus Z$ with $z_m = M$, and let $U := \{v\in V(G) : z_v>M\} \subseteq Z$. Apply Equation~\eqref{eq:grounded_laplac_sum} with $S = U \cup \{m\}$. We have $z_v \geq M$ when $v \in S$ and $z_v \leq M$ when $v \in V(G) \setminus S$, so the second sum is nonnegative. The first sum equals $d_{m}M$, so we get
    \begin{align*}
        M \leq d_{m}M \leq |S| \leq k + 1.
    \end{align*}

    We now give names to the vertices in $U$. Write $U = \{v_1, v_2, \ldots, v_\ell\}$, ordered so that $x_i := z_{v_i} - M$ satisfies $x_1 \geq x_2 \geq \cdots \geq x_\ell > 0$, and set $x_{\ell+1} := 0$. For each $1 \leq j \leq \ell$, substitute $S = \{v_1, v_2, \ldots, v_j\}$ into Equation~\eqref{eq:grounded_laplac_sum}. The first sum equals zero, and every summand in the second sum is nonnegative. At least one edge leaves $S$, since every connected component of $G$ contains a vertex outside $Z$. If $\{v_i,w\}$ is such an edge, with $v_i \in S$ and $w \notin S$, then
    \[
        z_{v_i} - z_w = (x_i + M) - z_w \geq (x_j + M) - (x_{j+1} + M) = x_j - x_{j+1}.
    \]
    It follows that $x_j - x_{j+1} \leq |S| = j$, so
    \begin{align*}
        \sum_{j=1}^{\ell}x_j = \sum_{j=1}^{\ell}j(x_j - x_{j+1}) \leq \sum_{j=1}^{\ell}j^2 \leq \sum_{j=1}^{k}j^2 = \frac{k(k+1)(2k+1)}{6}.
    \end{align*}
    We thus have, for every $W \subseteq V(G)$,
    \begin{align*}
        \sum_{v\in W}z_v \leq M|W|+\sum_{v\in W\cap U}(z_v-M) \leq (k+1)|W|+\frac{k(k+1)(2k+1)}{6},
    \end{align*}
    which proves part~(d).

    Finally, summing all equations in $K\z=\one$ gives $\sum_{v\in V(G)}d_vz_v=n$. Since $\z\geq\0$ and $d_v\geq1$ outside $Z$, we have $\sum_{v\notin Z}z_v\leq n$. Applying part~(d) with $W=Z$ therefore gives
    \begin{align*}
        \sum_{v\in V(G)}z_v \leq n + k(k+1) + \frac{k(k+1)(2k+1)}{6} = n + \frac{k(k+1)(2k+7)}{6},
    \end{align*}
    which proves part~(c).
\end{proof}

\subsection{A maximum at a vertex of degree 3}

We next use the bound from Lemma~\ref{lem:grounded}(d) to show that, under a restriction on the common neighbors of $N(v)$, a vertex $v$ of degree $3$ has a neighbor of very high degree.

\begin{lemma}\label{lem:boundary}
    Let $G$ be an eigensimple graph on $n \geq 32$ vertices and let $v \in V(G)$ have $d(v) = 3$. If at most two vertices $w \notin N(v) \cup \{v\}$ have $N(v) \subseteq N(w)$, then some neighbor of $v$ has degree at least $14$.
\end{lemma}

\begin{proof}
    Define $S := N(v)$ and $T := V(G) \setminus (S \cup \{v\})$, so that $V(G)$ is the disjoint union of $\{v\}$, $S$, and $T$. Let $B$ be the $|T| \times |S|$ matrix with rows indexed by $T$ and columns indexed by $S$, whose $(i,j)$-entry is $1$ if $\{i,j\} \in E(\Gc)$ and $0$ otherwise. In particular, notice that $[B\one]_i = |\overline{N}(i) \cap S| \geq 0$.
    
    For every connected component $C$ of $\Gc[T]$, there is an edge of $\Gc$ joining a vertex of $C$ to a vertex of $S$, since otherwise $v$ would be a cut vertex of $\Gc$, contradicting Lemma~\ref{lem:eigensimple_properties}(b). It follows that every connected component of $\Gc[T]$ has a vertex $i$ for which $[B\one]_i = |\overline{N}(i) \cap S| > 0$, so
    \begin{align*}
        K := L(\Gc[T]) + \diag(B\one)
    \end{align*}
    is a grounded Laplacian. Furthermore, we have $[B\one]_i = 0$ if and only if $S \subseteq N(i)$. By hypothesis, there are at most two such vertices $i$. We can thus apply Lemma~\ref{lem:grounded}(d) with $k \leq 2$ and $\z := K^{-1}\one$ to see that
    \begin{align}\label{ineq:3wp5}
        \sum_{i\in W}z_i \leq 3|W| + 5 \ \ \text{for every} \ \ W \subseteq T.
    \end{align}

    Now choose a $1$-eigenvector $\y$ of $\Gc$. Our next goal is to derive an upper bound on $|y_i|$ for $i \in T$. Since $T \subseteq \overline{N}(v)$, every vertex $i \in T$ satisfies
    \begin{align*}
        [K\y_T]_i & = \big[L(\Gc[T])\y_T\big]_i + \big[\diag(B\one)\y_T\big]_i \\
        & = \left(\big|\overline{N}(i) \cap T\big|y_i - \sum_{j\in\overline{N}(i)\cap T}y_j\right) + \big|\overline{N}(i)\cap S\big|y_i \\
        & = \Big(\big|\overline{N}(i)\big| - 1\Big)y_i - \sum_{j\in\overline{N}(i)\cap T}y_j \\
        & = \sum_{j\in\overline{N}(i)}y_j - \sum_{j\in\overline{N}(i)\cap T}y_j \\
        & = y_v + \sum_{j\in\overline{N}(i)\cap S}y_j = y_v + \big[B\y_S\big]_i,
    \end{align*}
    where the fourth equality comes from the eigenvector equation (Equation~\eqref{eq:eigenvector_equation}) applied at $i$, and the fifth equality comes from the fact that $V(G)$ is the disjoint union of $\{v\}$, $S$, and $T$. Since this holds for all $i \in T$, we have
    \begin{align}\label{eq:yu_proof_b}
        K\y_T = y_v\one + B\y_S.
    \end{align}
    Lemma~\ref{lem:eigensimple_properties}(e), applied to $\Gc$ and an arbitrary $u \in V(G)$, then gives
    \begin{align}\label{eq:yu_proof}
         \big(n-1-d(u)\big)y_u = -\sum_{w \in N(u)} y_w.
    \end{align}
    Suppose (for the sake of establishing a contradiction) that $\y_S = \0$. Then the right-hand side of Equation~\eqref{eq:yu_proof} equals zero when $u = v$, so $y_v = 0$. Plugging that into Equation~\eqref{eq:yu_proof_b} then shows that $K\y_T = \0$. Since $K$ is invertible by Lemma~\ref{lem:grounded}(a), this implies $\y_T = \0$, so $\y = \0$, contradicting the fact that $\y$ is an eigenvector.
    
    It follows that $\y_S \neq \0$. We can thus rescale $\y$ so that $\max_{u \in S}\big\{|y_u|\big\} = 1$. Equation~\eqref{eq:yu_proof} with $u = v$ then gives
    \begin{align}\label{eq:yu_proof_d}
        (n-4)|y_v| = \left|\sum_{w \in S}y_w\right| \leq 3.
    \end{align}
    Furthermore, since Laplacian matrices have row sums equal to zero, we have $K\one = B\one$, so $K^{-1}B\one = \one$. By Lemma~\ref{lem:grounded}(b), $K^{-1} \geq 0$ entrywise, so $K^{-1}B \geq 0$ entrywise and each row of $K^{-1}B$ sums to $1$. Multiplying Equation~\eqref{eq:yu_proof_b} on the left by $K^{-1}$ therefore gives, for every $i \in T$,
    \begin{align}\label{eq:yu_proof_c}
        |y_i| \leq |y_v|[K^{-1}\one]_i + \big|[K^{-1}B\y_S]_i\big| \leq z_i|y_v| + \sum_{w \in S}[K^{-1}B]_{i,w}|y_w| \leq z_i|y_v| + 1.
    \end{align}

    Now choose $u \in S$ with $|y_u| = 1$ (recall that we scaled $\y$ so that $\max_{u \in S}\big\{|y_u|\big\} = 1$, so such a $u$ exists) and define $W := N(u) \cap T$. Of the $d(u)$ neighbors of $u$, one is $v$, exactly $|W|$ lie in $T$, and the other $d(u) - |W| - 1$ lie in $S$. Using Equation~\eqref{eq:yu_proof}, followed by Inequalities~\eqref{eq:yu_proof_c}, \eqref{eq:yu_proof_d}, and~\eqref{ineq:3wp5} (in that order) then gives
    \begin{align*}
        n-1-d(u) & \leq \sum_{w\in N(u)}|y_w| \\
        & \leq |y_v| + \big(d(u) - |W| - 1\big) + \sum_{i \in W}|y_i| \\
        & \leq |y_v| + \big(d(u) - |W| - 1\big) + \sum_{i \in W}\big(z_i|y_v| + 1\big) \\
        & = d(u) - 1 + |y_v|\left(1+\sum_{i\in W}z_i\right) \\
        & \leq d(u) - 1 + \frac{3(3|W|+6)}{n-4} \leq d(u) - 1 + \frac{3(3d(u)+3)}{n-4},
    \end{align*}
    where the final inequality uses $|W| \leq d(u) - 1$. Rearranging gives
    \begin{align*}
        n \leq 2d(u) + \frac{3(3d(u)+3)}{n-4}.
    \end{align*}
    If $d(u) \leq 13$ then $n \geq 32$ gives $n \leq 26 + 126/28 = 61/2 < 32$, which is a contradiction that completes the proof.
\end{proof}

We now apply this lemma to a vertex of $G$ where its $1$-eigenvector attains its maximum. Its eigenvector equation first gives the required restriction on common neighbors. The neighbor of degree at least $14$ then provides an additional bound on its entries.

\begin{proof}[Proof of the degree-$3$ case of Lemma~\ref{lem:max_513_bound}]
    Suppose that $x_p = \max(\x)$ for some vertex $p$ of $G$ with $d(p) = 3$. Rescale $\x$ so that $x_p = \max(\x) = 1$, making it our objective to show that $\|\x_{+}\|^2 \geq 13/5$.

    The eigenvector equation (Equation~\eqref{eq:eigenvector_equation}) at $p$ gives $\sum_{w\in N(p)} x_w = 2$. Each of these three $x_w$ entries is at most $1$, so each is also nonnegative (since if one were negative then the other two entries could only contribute at most $2$ to the sum). Choose $u \in N(p)$ with largest entry $M := x_u \geq 2/3$, and set $U := N(u) \setminus (N(p)\cup\{p\})$ and $m := |U|$. The $d(u) - m - 1$ neighbors $w$ of $u$ in $N(p)$ all have $x_w \leq M$. The eigenvector equation at $u$ therefore gives
    \begin{align*}
        \big(d(u)-1\big)M = \sum_{w \in N(u)}x_w = 1 + \sum_{w \in N(u) \cap N(p)}x_w + \sum_{w \in U}x_w \leq 1 + \big(d(u) - m - 1\big)M + \sum_{w\in U}x_w,
    \end{align*}
    and hence $\sum_{w \in U}x_w \geq mM - 1$.\smallskip

    \noindent\underline{\textbf{Case 1:} $m \geq 3$.} Then
    \begin{align*}
        \sum_{\substack{w\in U\\x_w>0}}x_w \geq \sum_{w \in U}x_w \geq mM - 1 \geq \frac{2m}{3} - 1 \geq \frac{m}{3}.
    \end{align*}
    Applying the Cauchy--Schwarz inequality to the two sums in the inequality
    \begin{align*}
        \|\x_{+}\|^2 \geq 1 + \sum_{w \in N(p)}x_w^2 + \sum_{\substack{w\in U\\x_w>0}}x_w^2,
    \end{align*}
    and recalling that $\sum_{w\in N(p)}x_w = 2$, gives
    \begin{align*}
        \|\x_{+}\|^2 \geq 1 + \frac{2^2}{3} + \frac{(m/3)^2}{m} \geq \frac{8}{3} \geq \frac{13}{5},
    \end{align*}
    as desired.\smallskip

    \noindent\underline{\textbf{Case 2:} $m \leq 2$.} Every vertex $w \notin N(p) \cup \{p\}$ with $N(p) \subseteq N(w)$ is adjacent to $u$, since $u \in N(p)$, and therefore belongs to $U$. Since $|U| = m \leq 2$, Lemma~\ref{lem:boundary} tells us that there exists a vertex $v \in N(p)$ with degree $d(v) \geq 14$.

    Since $v \in N(p)$, we have $0 \leq x_v \leq 1$. Define
    \begin{align*}
        P := 1 + \sum_{w \in N(p)}x_w^2 = \sum_{w \in N(p) \cup \{p\}}x_w^2 \leq \|\x_{+}\|^2.
    \end{align*}
    We have $\sum_{w \in N(p) \setminus \{v\}}x_w = 2 - x_v$, so the Cauchy--Schwarz inequality gives
    \begin{align}\label{ineq:P_t_bound}
        \|\x_{+}\|^2 \geq P \geq 1 + x_v^2 + \frac{(2-x_v)^2}{2} = \frac{1}{2}\big(3x_v^2 - 4x_v + 6\big).
    \end{align}
    If $x_v \leq 3/14$ then the quadratic on the right exceeds $13/5$, so we assume from now on that $x_v > 3/14$.
    
    Set $W := N(v) \setminus (N(p) \cup \{p\})$. Since $p \in N(v)$, we have $|W| \leq d(v) - 1$. Furthermore, $x_w \geq 0$ when $w \in N(p)$, and $W$ is disjoint from $N(p) \cup \{p\}$, so
    \begin{align*}
        \sum_{\substack{w \in W\\x_w>0}}x_w^2 \leq \|\x_{+}\|^2 - P.
    \end{align*}
    If we write $N(v)$ as the disjoint union of $\{p\}$, $N(v) \cap N(p)$, and $W$, then the eigenvector equation at $v$ gives
    \begin{align}\label{eq:t_eigvec}
        \big(d(v)-1\big)x_v & = 1 + \sum_{w \in N(v) \cap N(p)}x_w + \sum_{w \in W}x_w.
    \end{align}
    Since $N(v) \cap N(p) \subseteq N(p) \setminus \{v\}$, and $\sum_{w \in N(p) \setminus \{v\}}x_w = 2 - x_v$, the central sum in Equation~\eqref{eq:t_eigvec} is no larger than $2 - x_v$. Discarding the nonpositive terms from the rightmost sum and then applying the Cauchy--Schwarz inequality gives
    \begin{align*}
        \big(d(v)-1\big)x_v & \leq 3 - x_v + \sum_{\substack{w \in W\\x_w>0}}x_w \\
        & \leq 3 - x_v + \sqrt{|W|\sum_{\substack{w \in W\\x_w>0}}x_w^2} \\
        & \leq 3 - x_v + \sqrt{\big(d(v)-1\big)\big(\|\x_{+}\|^2-P\big)}.
    \end{align*}
    
    Since $d(v) \geq 14$ and $x_v > 3/14$, rearranging this inequality gives
    \begin{align*}
        0 < d(v)x_v - 3 \leq \sqrt{(d(v)-1)\big(\|\x_{+}\|^2 - P\big)}.
    \end{align*}
    Squaring and dividing both sides by $d(v) - 1 > 0$ gives
    \begin{align*}
        \|\x_{+}\|^2 \geq P + \frac{(d(v)x_v-3)^2}{d(v)-1} \geq \frac{1}{2}\big(3x_v^2 - 4x_v + 6\big) + \frac{(d(v)x_v-3)^2}{d(v)-1},
    \end{align*}
    where the final inequality comes from Inequality~\eqref{ineq:P_t_bound}. Since $x_v > 3/14$ and $d(v) \geq 14$, it is straightforward to verify that the right-hand side is at least $1126/431$, which is larger than $13/5$. This completes the proof.
\end{proof}

\subsection{A maximum at a vertex of degree 2}\label{sec:degree_two_neighborhoods}

All that remains in order to prove Lemma~\ref{lem:max_513_bound} is the degree $2$ case. We begin this case with one more helper lemma.

\begin{lemma}\label{lem:degree_two_schur}
    Let $G$ be an eigensimple graph on $n \geq 32$ vertices, let $p \in V(G)$ have $d(p) = 2$, and label its neighbors by $a$ and $b$. Suppose that $b$ has at most seven neighbors outside $\{p,a\}$ and that some vertex outside $\{p,a,b\}$ is nonadjacent to both $a$ and $b$. Then $a$ and $b$ are adjacent, and at most two vertices outside $\{p,a,b\}$ are nonadjacent to both of them.
\end{lemma}

\begin{proof}
    Set $T := V(G) \setminus \{p,a,b\}$ and $m := |T| = n-3$, and let $e \in \{0,1\}$ indicate whether or not $a$ and $b$ are adjacent (a value of $1$ indicates adjacency, while a value of $0$ indicates non-adjacency). Let $\p,\w \in \{0,1\}^T$ indicate adjacency in $\Gc$ to $a$ and $b$, respectively (e.g., $p_x = 1$ if $x \in \overline{N}(a)$). As in the proof of Lemma~\ref{lem:boundary}, the matrix
    \begin{align*}
        K := L\big(\Gc[T]\big) + \diag(\p+\w)
    \end{align*}
    is a grounded Laplacian, since $p$ is not a cut vertex of $\Gc$. The vector $\w$ has at most seven zero entries, so the same is true of $\p+\w$. Lemma~\ref{lem:grounded}(a)--(c) tells us that $K$ is invertible, $K^{-1} \geq 0$ entrywise, and
    \begin{align*}
        \one^TK^{-1}\one \leq m+196.
    \end{align*}
    Define $\z := K^{-1}\p$ and $g := \w^T\z$. Since $K\one = \p+\w$, we have $K^{-1}\w = \one - \z$ and thus $\0 \leq \z \leq \one$ entrywise.

    Choose a $1$-eigenvector $\y$ of $\Gc$, and define $\delta := y_a-y_b$. The eigenvector equations at the vertices in $T$ give
    \begin{align*}
        K\y_T = y_p\one + y_a\p + y_b\w, \qquad \text{so} \qquad \y_T = y_pK^{-1}\one + y_b\one + \delta\z.
    \end{align*}
    The equation at $p$ gives $\one^T\y_T = (m-1)y_p$, and combining this with $\y \perp \one$ gives $y_a+y_b = -my_p$. Summing the preceding expression for $\y_T$ and substituting $y_b = (-my_p-\delta)/2$ therefore gives
    \begin{align*}
        \left(\frac{m^2}{2}+m-1-\one^TK^{-1}\one\right)y_p = \left(\one^T\z - \frac{m}{2}\right)\delta.
    \end{align*}
    Since $m \geq 29$, the coefficient of $y_p$ is at least $m^2/2-197 > m^2/4$, while $|\one^T\z-m/2| \leq m/2$. Also, $\delta \neq 0$, since otherwise these equations would give $y_p = y_a = y_b = 0$ and then $\y_T = \0$. It follows that $|y_p| < 2|\delta|/m$.

    The eigenvector equation at $a$ is
    \begin{align*}
        (\one^T\p-e)y_a = \p^T\y_T+(1-e)y_b.
    \end{align*}
    Substituting the expression for $\y_T$, using $\p^TK^{-1}\one = \one^T\z$ and $\p^T\z = \one^T\p-g$, and again substituting for $y_b$, gives
    \begin{align*}
        (2g+1-2e)\delta = (2\one^T\z-m)y_p.
    \end{align*}
    Since $|2\one^T\z-m| \leq m$ and $|y_p| < 2|\delta|/m$, we obtain $2g+1-2e < 2$.

    Finally, symmetry of $K^{-1}$ gives $g = \p^T(\one-\z) = \w^T\z$, so
    \begin{align*}
        2g = \sum_{i \in T}\big(p_i(1-z_i) + w_i z_i\big) \geq \big|\supp(\p) \cap \supp(\w)\big|.
    \end{align*}
    Indeed, every summand is nonnegative, and each vertex in both supports contributes $1$. By hypothesis, this intersection is nonempty, so $2g \geq 1$. The inequality $2g+1-2e < 2$ therefore forces $e = 1$, and then gives $\big|\supp(\p) \cap \supp(\w)\big| < 3$, which completes the proof.
\end{proof}

\begin{proof}[Proof of the degree-$2$ case of Lemma~\ref{lem:max_513_bound}]
    Choose vertices $p$ and $s$ of $G$ as in Lemma~\ref{lem:extremal_path} (applied to $\x/\|\x\|$) so that $x_p = \max(\x)$, and assume that $d(p) = 2$. Rescale $\x$ so that $x_p = \max(\x) = 1$, making it our objective to show that $\|\x_{+}\|^2 \geq 13/5$.

    Write $N(p) = \{a,b\}$. The eigenvector equation at $p$ gives $x_b+x_a = 1$. Since neither entry exceeds $1$, both must be nonnegative, so we can label the vertices so that $x_b = t$ and $x_a = 1-t$ for some $t \in [1/2, 1]$. Define $U := N(b) \setminus \{a,p\}$ and $k := |U|$. If $a \in N(b)$, then $x_a \leq t$, so the eigenvector equation at $b$ gives
    \begin{align}\label{ineq:deg2_proof1}
        \sum_{i \in U}x_i \geq kt-1 \geq \frac{k-2}{2}.
    \end{align}
    The same bound holds if $a \notin N(b)$, since then the first inequality is an equality.\smallskip

    \noindent\underline{\textbf{Case 1:} $k \geq 8$.} Omitting the nonpositive terms from the sum in Inequality~\eqref{ineq:deg2_proof1} and applying the Cauchy--Schwarz inequality gives
    \begin{align*}
        \|\x_{+}\|^2 \geq 1+t^2+(1-t)^2 + \frac{1}{k}\left(\sum_{\substack{i \in U\\x_i>0}}x_i\right)^2 \geq \frac{3}{2}+\frac{(k-2)^2}{4k} \geq \frac{13}{5},
    \end{align*}
    as desired.\smallskip

    \noindent\underline{\textbf{Case 2:} $k \leq 7$.} Since the distance from $p$ to $s$ is $3$, the vertex $s$ is nonadjacent to both $a$ and $b$. Lemma~\ref{lem:degree_two_schur} therefore tells us that $a$ and $b$ are adjacent to each other and have at most two common nonneighbors outside $\{p,a,b\}$. Define $r := d(b) = k+2 \leq 9$. If $r = 2$, deleting $a$ would separate $\{p,b\}$ from the remaining vertices, contradicting Lemma~\ref{lem:eigensimple_properties}(b), so $3 \leq r \leq 9$. Every nonneighbor of $a$ either belongs to $U$ or is one of the at most two common nonneighbors, so $\overline{d}(a) \leq (r-2)+2 = r$. It follows that $d(a) \geq n-r-1 \geq 31-r$.

    We now obtain a lower bound and an upper bound on $\|\x\|^2$. Define $W := N(a) \setminus \{p,b\}$ and $C := U \cap W$, and let $Z$ be the set of common nonneighbors of $a$ and $b$ outside $\{p,a,b\}$. Then $|C|+|Z| \leq (r-2)+2 = r$. The eigenvector equations at $a$ and $b$ give
    \begin{align}\label{eq:deg2_proof2}
        \sum_{i \in U}x_i = rt-2 \qquad \text{and} \qquad \sum_{i \in W}x_i = d(a)(1-t)-2.
    \end{align}
    Since $\x \perp \one$ and $x_p+x_a+x_b = 2$, the sum of $x_w$ for $w \in V(G) \setminus \{p,a,b\}$ equals $-2$. Adding the two sums~\eqref{eq:deg2_proof2} counts the entries on $C$ twice and those on $Z$ zero times, so
    \begin{align*}
        \sum_{i \in C}x_i - \sum_{i \in Z}x_i = rt + d(a)(1-t) - 2.
    \end{align*}
    The Cauchy--Schwarz inequality and $|C|+|Z| \leq r$ therefore give
    \begin{align*}
        \big(rt+d(a)(1-t)-2\big)^2 & \leq r\sum_{i \in C \cup Z}x_i^2 \leq r\big(\|\x\|^2-1-t^2-(1-t)^2\big).
    \end{align*}
    The expression being squared is positive. Using $d(a) \geq 31-r$, we obtain
    \begin{align}\label{ineq:def2_x_lb}
        \|\x\|^2 \geq 1+t^2+(1-t)^2+\frac{\big(r-2+(31-2r)(1-t)\big)^2}{r}.
    \end{align}

    For the upper bound, apply Lemma~\ref{lem:eigensimple_properties}(d) to $\x/\|\x\|$ and multiply the resulting positive semidefinite matrix on the right by $\e_b + (r-1)\e_p$ and on the left by the transpose of that vector. Since $d(p) = 2$, $d(b) = r$, and $p$ and $b$ are adjacent, this gives
    \begin{align}\label{ineq:def2_x_ub}
        0 \leq -r+\frac{2r^2}{n}+\frac{(r-1+t)^2}{\|\x\|^2}, \qquad \text{so} \qquad \|\x\|^2 \leq \frac{16(r-1+t)^2}{r(16-r)},
    \end{align}
    where we used $n \geq 32$ to get the second inequality. Inequalities~\eqref{ineq:def2_x_lb} and~\eqref{ineq:def2_x_ub} imply that $f_r(t) \leq 0$, where
    \begin{align*}
        f_r(t) := 1+t^2+(1-t)^2 + \frac{\big(r-2+(31-2r)(1-t)\big)^2}{r} - \frac{16(r-1+t)^2}{r(16-r)}.
    \end{align*}
    This function decreases on $[1/2,1]$. Indeed, $3 \leq r \leq 9$ gives
    \begin{align*}
        f_r^\prime(t) & = 4t-2-\frac{2(31-2r)\big(r-2+(31-2r)(1-t)\big)}{r} - \frac{32(r-1+t)}{r(16-r)} \leq 2-\frac{26}{3} < 0.
    \end{align*}
    Define $\tau_3 := 19/20 \geq 1/2$ and $\tau_r := (18-r)/16 \geq 1/2$ for $4 \leq r \leq 9$. Direct substitution shows that $f_r(\tau_r) > 0$ for all $3 \leq r \leq 9$, so $f_r(t) \leq 0$ forces $t > \tau_r$.

    If $r = 3$, then $U$ has just one vertex $c$, and $x_c = 3t-2 > 17/20$. Hence
    \begin{align*}
        \|\x_{+}\|^2 \geq 1+t^2+x_c^2 > 1+\left(\frac{19}{20}\right)^2+\left(\frac{17}{20}\right)^2 = \frac{21}{8} \geq \frac{13}{5},
    \end{align*}
    as desired. If $4 \leq r \leq 9$, then
    \begin{align*}
        \sum_{i \in U}x_i = rt - 2 > r\tau_r-2 = \frac{(r-2)(16-r)}{16} > 0.
    \end{align*}
    Omitting the nonpositive terms and applying the Cauchy--Schwarz inequality, using $|U| = r-2$, gives
    \begin{align*}
        \|\x_{+}\|^2 & \geq 1+t^2+(1-t)^2+\frac{1}{r-2}\left(\sum_{\substack{i \in U\\x_i>0}}x_i\right)^2 \geq \frac{3}{2} + \frac{(r-2)(16-r)^2}{256} \geq \frac{13}{5},
    \end{align*}
    as desired. Since both cases give the desired bound, this completes the proof of the degree-$2$ case, and thus (finally!) of Lemma~\ref{lem:max_513_bound} as a whole.
\end{proof}

\section{The eleven inequalities}\label{app:certificates}

We now prove Lemma~\ref{lem:eleven_inequalities}. We first bound $\kappa_*$ (from Definition~\ref{defn:kappa_func}) from below by polynomials of degree at most $2$ in each variable. Substituting these bounds into the definitions of $F_{\textup{A}}$, $F_{\textup{B}}$, and three of the functions $F_{i,j}$ from Definition~\ref{defn:exclusion_functions} gives five polynomials, each of which can be written as a quadratic form with a positive definite matrix. This proves five of the desired inequalities, and the other six then follow by symmetry and interpolation.

\subsection{Biquadratic lower bounds for \texorpdfstring{$\kappa_*$}{kappa}}

We seek polynomial lower bounds for $\kappa_*$ of degree at most $2$ in each variable. We start by giving names to some pieces that will appear in these lower bounds. For $\rho, \sigma, \tau \in [0,1]$ and $a,b,x,y \in \R$, define
\begin{align*}
    \eta_\tau(a) & := \frac{(32a-1)\tau}{1+30\tau^2}, \\
    \zeta_\tau(a,b) & := \big(a-\eta_\tau(a)b\big)^2 + \frac{\big((1-2a)b+\eta_\tau(a)\big)^2}{30}, \\
    k_{\sigma,\tau}(x,y) & := 1-4x(1-x)y(1-y)-\zeta_\tau(x,y)-\zeta_{1-\tau}(1-x,1-y) \\
    & \qquad \qquad \qquad \qquad \qquad \quad \, - \zeta_{1-\sigma}(y,1-x)-\zeta_\sigma(1-y,x).
\end{align*}
We also define $\widehat{k}_\rho$ to be the unique polynomial satisfying
\begin{align*}
    2(x-y)\widehat{k}_\rho(x,y) = k_{\rho,\rho}(x,y) - k_{\rho,\rho}(y,x).
\end{align*}
To see that this polynomial exists, we note that the polynomial $k_{\rho,\rho}(x,y) - k_{\rho,\rho}(y,x)$ equals $0$ when $x = y$, so $k_{\rho,\rho}(x,y) - k_{\rho,\rho}(y,x)$ has $(x-y)$ as a factor.

With this notation established, the following lemma provides the bounds that we will substitute into the functions $F_{\textup{A}}$, $F_{\textup{B}}$, and $F_{i,j}$.

\begin{lemma}\label{lem:kappa_polynomial_bounds}
    Suppose $\rho \in [0,1]$ and $\lambda > 0$. Then
    \begin{align*}
        K_{\textup{A}}(x,y) & := k_{1/3,1/2}(x,y) \quad \text{and} \\
        K_{\rho,\lambda}(x,y) & := \frac{k_{\rho,\rho}(x,y) + k_{\rho,\rho}(y,x)}{2} - \frac{\widehat{k}_{\rho}(x,y)^2}{4\lambda} - \lambda(x-y)^2
    \end{align*}
    are polynomials of degree at most $2$ in each variable, and
    \begin{align*}
        K_{\textup{A}}(x,y) \leq \kappa_{*}(x,y) \qquad \text{and} \qquad K_{\rho,\lambda}(x,y) \leq \kappa_{*}(x,y)
    \end{align*}
    for all $x,y \in [0,1/2]$.
\end{lemma}

\begin{proof}
    We first obtain an upper bound for each of the four rational terms that are subtracted in the definition of $\kappa$ (from Definition~\ref{defn:kappa_func}). For $a,b \in \R$ and $\tau \in [0,1]$, we will apply the Cauchy--Schwarz inequality to the vectors
    \begin{align}\label{eq:quadratic_kappa_two_vecs}
        \left(a-b\eta_\tau(a), \frac{(1-2a)b + \eta_\tau(a)}{\sqrt{30}}\right) \qquad \text{and} \qquad \big(1,\sqrt{30}b\big).
    \end{align}
    Their inner product is $a+(1-2a)b^2$, so applying Cauchy--Schwarz and then dividing both sides by $1+30b^2$ gives
    \begin{align}\label{eq:quadratic_kappa_bound}
        \frac{\big(a+(1-2a)b^2\big)^2}{1+30b^2} \leq \big(a-b\eta_\tau(a)\big)^2 + \frac{\big((1-2a)b+\eta_\tau(a)\big)^2}{30} = \zeta_\tau(a,b).
    \end{align}
    Furthermore, if $\tau = b$ then the two vectors in~\eqref{eq:quadratic_kappa_two_vecs} are scalar multiples of each other, so equality holds in the Cauchy--Schwarz inequality and thus in Inequality~\eqref{eq:quadratic_kappa_bound}.

    Applying Inequality~\eqref{eq:quadratic_kappa_bound} to the four rational terms in the definition of $\kappa(x,y)$ gives
    \begin{align}\label{ineq:k_sigma_kappa}
        k_{\sigma,\tau}(x,y) \leq \kappa(x,y)
    \end{align}
    for all $x,y \in [0,1]$ and $\sigma,\tau \in [0,1]$. Since $\eta_\tau(a)$ has degree at most $1$ in $a$, the polynomial $\zeta_\tau(a,b)$ has degree at most $2$ in each of $a$ and $b$. Consequently, $k_{\sigma,\tau}$ and $K_{\textup{A}}$ also have degree at most $2$ in each variable, $\widehat{k}_{\rho}$ has degree at most $1$ in each variable, and so $K_{\rho,\lambda}$ has degree at most $2$ in each variable.

    To prove that $K_{\rho,\lambda}(x,y) \leq \kappa_{*}(x,y)$, it suffices (because $K_{\rho,\lambda}(x,y)$ is symmetric with respect to $x$ and $y$ and because of Inequality~\eqref{ineq:k_sigma_kappa}) to show that $K_{\rho,\lambda}(x,y) \leq k_{\rho,\rho}(x,y)$. Indeed,
    \begin{align*}
        k_{\rho,\rho}(x,y) - K_{\rho,\lambda}(x,y) & = (x-y)\widehat{k}_{\rho}(x,y)+\lambda(x-y)^2+\frac{\widehat{k}_{\rho}(x,y)^2}{4\lambda} \\
        & = \frac{\big(2\lambda(x-y) + \widehat{k}_{\rho}(x,y)\big)^2}{4\lambda} \geq 0.
    \end{align*}
    
    To prove that $K_{\textup{A}}(x,y) \leq \kappa_{*}(x,y)$, first notice that $K_{\textup{A}}(x,y) = k_{1/3,1/2}(x,y) \leq \kappa(x,y)$ follows immediately from Inequality~\eqref{ineq:k_sigma_kappa}. It thus suffices to show that $K_{\textup{A}}(x,y) \leq \kappa(y,x)$. To this end, we compute
    \begin{align*}
        K_{\textup{A}}(x,1-y) - K_{\textup{A}}(x,y) = \frac{160(1-2y)}{90307009}\big(2033982x^2 - 1614470x + 325761\big),
    \end{align*}
    which is nonnegative for all $x,y \in [0,1/2]$. In particular, this means that $K_{\textup{A}}(x,y) \leq K_{\textup{A}}(x,1-y)$ holds for $x,y \in [0,1/2]$. By then using Inequality~\eqref{ineq:k_sigma_kappa} and Lemma~\ref{lem:kappa_properties}, we see that
    \begin{align*}
        K_{\textup{A}}(x,y) & \leq K_{\textup{A}}(x,1-y) = k_{1/3,1/2}(x,1-y) \leq \kappa(x,1-y) = \kappa(y,x),
    \end{align*}
    which completes the proof.
\end{proof}

\subsection{Matrix representations of the biquadratic polynomials}\label{sec:matrix_rep}

We now use the Gram matrix method for sums of squares of polynomials (see, for example, \cite[Theorem~1 and Example~1]{PW98} and \cite[Section~3.2]{Par03}) to show that the biquadratic lower bounds arising from Lemma~\ref{lem:kappa_polynomial_bounds} on the functions $F_{\textup{A}}$, $F_{\textup{B}}$, and $F_{i,j}$ are all strictly positive on their domains, so $F_{\textup{A}}$, $F_{\textup{B}}$, and $F_{i,j}$ are all strictly positive on their domains as well, thus proving Lemma~\ref{lem:eleven_inequalities}.

Explicitly, we will plug the bounds of Lemma~\ref{lem:kappa_polynomial_bounds} into the functions from Definition~\ref{defn:exclusion_functions}. For $F_{\textup{B}}$ and each $F_{i,j}$ we have a bit of freedom in how we do this, since we can choose $\rho \in [0,1]$ and $\lambda \in (0,\infty)$ as we see fit. Making the choices specified in Table~\ref{tab:gram_parameters} gives us the inequalities
\begin{align*}
    F_{\textup{A}}(x,y) & \geq R\left(\frac{7}{43},x\right) R\left(\frac{4}{13},y\right) - R\left(\frac{7}{43},x\right) - R\left(\frac{4}{13},y\right) + K_{\textup{A}}(x,y), \\
    F_{\textup{B}}(x,y) & \geq \frac{25}{13}R\left(\frac{7}{43},x\right)R\left(\frac{5}{13},y\right) - 2R\left(\frac{5}{13},y\right) - \big(1+T_{*}(y)\big)R\left(\frac{7}{43},x\right) + K_{\tfrac{1}{3},\tfrac{1}{4}}(x,y), \\
    F_{3,3}(x,y) & \geq \frac{377}{196}R\left(\frac{5}{13},x\right)R\left(\frac{5}{13},y\right) - \big(1+T_3(x)\big)R\left(\frac{5}{13},y\right) \\
    & \qquad \qquad \qquad \qquad \qquad \qquad \ \, - \big(1+T_3(y)\big)R\left(\frac{5}{13},x\right) + K_{\tfrac{1}{3},\tfrac{3}{2}}(x,y),
\end{align*}
for example (and a similar lower bound on each of $F_{6,3}(x,y)$ and $F_{6,6}(x,y)$, which we omit for brevity). In each case, denote the biquadratic polynomial that serves as the lower bound by $\widetilde{F}$ with the same subscripts (e.g., $F_{\textup{A}}(x,y) \geq \widetilde{F}_{\textup{A}}(x,y)$ and $F_{6,3}(x,y) \geq \widetilde{F}_{6,3}(x,y)$). We will just need these five inequalities; the other six will follow via symmetry and an interpolation argument.

To prove that each $\widetilde{F}$ is positive, write
\begin{align*}
    \widetilde{F}(x,y)=\sum_{i,j=0}^2 c_{ij}x^iy^j.
\end{align*}
We express this polynomial as a quadratic form in the nonzero vector $(1,x,y,xy)$. For any $\theta \in \R$, the matrix
\begin{align}\label{eq:gram-matrix}
    M(\theta,\widetilde{F}) := \begin{bmatrix}
        c_{00} & c_{10}/2 & c_{01}/2 & \theta \\
        c_{10}/2 & c_{20} & c_{11}/2-\theta & c_{21}/2 \\
        c_{01}/2 & c_{11}/2-\theta & c_{02} & c_{12}/2 \\
        \theta & c_{21}/2 & c_{12}/2 & c_{22}
    \end{bmatrix}
\end{align}
satisfies
\begin{align*}
    \widetilde{F}(x,y) = [1,x,y,xy] M(\theta,\widetilde{F}) [1,x,y,xy]^T.
\end{align*}
(We note that $\theta$ can be chosen freely, and different choices lead to different matrices $M(\theta,\widetilde{F})$ that all correspond to the same biquadratic function $\widetilde{F}$.) If $M(\theta,\widetilde{F})$ is positive definite, then $\widetilde{F}(x,y) > 0$ for all real $x,y$, since $(1,x,y,xy)$ is never zero.

\begin{table}[htbp]
    \centering
    \begin{tabular}{@{}cccc@{}}\toprule
        Function & $\rho$ & $\lambda$ & $\theta$ \\\midrule
        $F_{\textup{A}}$ & -- & -- & $2$ \\
        $F_{\textup{B}}$ & $1/3$ & $1/4$ & $15/4$ \\
        $F_{3,3}$ & $1/3$ & $3/2$ & $8/3$ \\
        $F_{6,3}$ & $2/7$ & $1/2$ & $1$ \\
        $F_{6,6}$ & $2/9$ & $3$ & $3$ \\\bottomrule
    \end{tabular}
    \caption{The parameters used for the five polynomial lower bounds and their matrix representations. The function $F_{\textup{A}}$ uses $K_{\textup{A}}$, while the other functions use $K_{\rho,\lambda}$.}\label{tab:gram_parameters}
\end{table}

For each choice of subscript $\textup{s} \in \{\textup{A}, \textup{B}, (3,3), (6,3), (6,6)\}$ (i.e., each row of Table~\ref{tab:gram_parameters}), let $M_{\textup{s}}$ denote the matrix obtained from $\widetilde{F}_{\textup{s}}$ using the listed value of $\theta$. Then
\begin{align}\label{eq:five_explicit_identities}
    \widetilde{F}_{\textup{s}}(x,y) = [1,x,y,xy]M_{\textup{s}}[1,x,y,xy]^T.
\end{align}
These five matrices are displayed in Section~\ref{app:five-matrices}. They are all positive definite, which can be verified via exact integer arithmetic, e.g., via Sylvester's criterion.

\subsection{Proof of Lemma~\ref{lem:eleven_inequalities}}\label{sec:proof_of_lem_five}

\begin{proof}
    For each of the five matrices $M_{\textup{s}}$ from Section~\ref{sec:matrix_rep}, positive definiteness and Equation~\eqref{eq:five_explicit_identities} give $\widetilde{F}_{\textup{s}}(x,y) > 0$ for all real $x$ and $y$. As a result,
    \begin{align*}
        F_{\textup{s}}(x,y) \geq \widetilde{F}_{\textup{s}}(x,y) > 0,
    \end{align*}
    for all $x,y \in [0,1/2]$. This proves the desired inequalities for $F_{\textup{A}}$, $F_{\textup{B}}$, $F_{3,3}$, $F_{6,3}$, and $F_{6,6}$. Symmetry of $\kappa_{*}$ gives $F_{i,j}(x,y) = F_{j,i}(y,x)$, so $F_{3,6}(x,y) = F_{6,3}(y,x) > 0$ for $x,y \in [0,1/2]$ as well.
    
    All that's left is to prove the lemma in the case where at least one of $i$ or $j$ equals $4$. For $i \in \{3,4,6\}$, define
    \begin{align*}
        B_i(t) := 1+T_{i}(t)+\frac{10i}{(31-i)^2}R(\delta_i,t).
    \end{align*}
    Equation~\eqref{eq:Fij} then becomes
    \begin{align}\label{eq:Fi4_alt}
        F_{i,j}(x,y) & = 2R(\delta_i,x)R(\delta_j,y)-B_i(x)R(\delta_j,y)  -B_j(y)R(\delta_i,x)+\kappa_*(x,y).
    \end{align}
    Direct substitution then gives
    \begin{align}\label{eq:B_func_1}
        R(\delta_4,t) & = \frac{19}{32}R(\delta_3,t)+\frac{13}{32}R(\delta_6,t) \quad \text{and} \\\label{eq:B_func_2}
        \frac{19}{32}B_{3}(t) + \frac{13}{32}B_{6}(t) - B_{4}(t) & = \frac{5342775832t^2 + 2671847872t + 1204426337}{177176160000} > 0
    \end{align}
    for all $t \in [0,1/2]$.
    
    To see how this applies to $F_{4,j}$, substitute Equation~\eqref{eq:B_func_1} into Equation~\eqref{eq:Fi4_alt} (with $i = 4$) and subtract $\frac{19}{32}F_{3,j}(x,y) + \frac{13}{32}F_{6,j}(x,y)$. All terms except those involving $B_i(x)$ cancel, giving
    \begin{align*}
        F_{4,j}(x,y) - \frac{19}{32}F_{3,j}(x,y) - \frac{13}{32}F_{6,j}(x,y) = \left(\frac{19}{32}B_{3}(x) + \frac{13}{32}B_{6}(x)-B_{4}(x)\right)R(\delta_j,y) > 0,
    \end{align*}
    where the final inequality comes from Inequality~\eqref{eq:B_func_2} and Lemma~\ref{lem:R_lower_bound}. Taking $j=3$ and $j=6$ proves positivity of $F_{4,3}$ and $F_{4,6}$, respectively. Symmetry then gives positivity of $F_{3,4}$ and $F_{6,4}$. Finally, taking $j=4$ gives
    \begin{align*}
        F_{4,4}(x,y) & > \frac{19}{32}F_{3,4}(x,y) + \frac{13}{32}F_{6,4}(x,y) > 0,
    \end{align*}
    which completes the proof.
\end{proof}

\subsection{The five matrices}\label{app:five-matrices}

The five matrices used above are obtained from Equation~\eqref{eq:gram-matrix} and Table~\ref{tab:gram_parameters}. We display a positive integer multiple of each matrix so that all displayed entries are integers. The subscripts agree with those of the corresponding functions $F$.\\

\noindent $189644718900 M_{\textup{A}}$ equals
    \begingroup\small\setlength{\arraycolsep}{4pt}\renewcommand{\arraystretch}{1.4}
    \begin{align*}
        \begin{bmatrix}
            120815619554 & -177288644624 & -274818584614 & 379289437800 \\
            -177288644624 & 1201633392464 & 194685364328 & -2082765580384 \\
            -274818584614 & 194685364328 & 856323372884 & -1165705337069 \\
            379289437800 & -2082765580384 & -1165705337069 & 5926672958174
        \end{bmatrix},
    \end{align*}
    \endgroup
    
\noindent $82021275303240 M_{\textup{B}}$ equals
    \begingroup\small\setlength{\arraycolsep}{4pt}\renewcommand{\arraystretch}{1.4}
    \begin{align*}
        \begin{bmatrix}
            108543337024871 & -315934327543247 & -129335926228376 & 307579782387150 \\
            -315934327543247 & 1386900622345400 & 213749953654692 & -1598885428915904 \\
            -129335926228376 & 213749953654692 & 369160809063494 & -627429049555112 \\
            307579782387150 & -1598885428915904 & -627429049555112 & 3769871327008304
        \end{bmatrix},
    \end{align*}
    \endgroup
    
\noindent $88993083704015400 M_{3,3}$ equals
    \begingroup\small\setlength{\arraycolsep}{4pt}\renewcommand{\arraystretch}{1.4}
    \begin{align*}
        \begin{bmatrix}
            81875338267090038 & -162097320390263983 & -162097320390263983 & 237314889877374400 \\
            -162097320390263983 & 378530708880066346 & 314236460972467436 & -601124656192963096 \\
            -162097320390263983 & 314236460972467436 & 378530708880066346 & -601124656192963096 \\
            237314889877374400 & -601124656192963096 & -601124656192963096 & 1573493036620054912
        \end{bmatrix},
    \end{align*}
    \endgroup

\noindent $418402579567658571290625 M_{6,6}$ equals
    \begingroup\fontsize{7}{9}\selectfont\setlength{\arraycolsep}{4pt}\renewcommand{\arraystretch}{1.4}
    \begin{align*}
        \begin{bmatrix}
            288449407601816109352061 & -764974209472936160026239 & -764974209472936160026239 & 1255207738702975713871875 \\
            -764974209472936160026239 & 2381978052577480416584247 & 2109696371747539491103334 & -5019592216704879857120126 \\
            -764974209472936160026239 & 2109696371747539491103334 & 2381978052577480416584247 & -5019592216704879857120126 \\
            1255207738702975713871875 & -5019592216704879857120126 & -5019592216704879857120126 & 19146429681821423664848428
        \end{bmatrix},
    \end{align*}
    \endgroup
    
\noindent and $37875028906318831386629925000 M_{6,3}$ equals
    \begingroup\fontsize{6}{8}\selectfont\setlength{\arraycolsep}{1.5pt}\renewcommand{\arraystretch}{1.4}
    \begin{align*}
        \resizebox{\textwidth}{!}{$\begin{bmatrix}
            13413866039000635305591963011 & -38923171546473333832912962067 & -24702968754802720638497826136 & 37875028906318831386629925000 \\
            -38923171546473333832912962067 & 261308903953344935639820834746 & 9372174748871669370292787884 & -295517487976494247254438648496 \\
            -24702968754802720638497826136 & 9372174748871669370292787884 & 91706870319091148437468234772 & -86799175713086464743131515384 \\
            37875028906318831386629925000 & -295517487976494247254438648496 & -86799175713086464743131515384 & 794404124317032351775237026992
        \end{bmatrix}.$}
    \end{align*}
    \endgroup
\end{document}